\documentclass[11pt,reqno]{amsart}

\usepackage{amsmath, amsfonts, amsthm, amssymb}
\usepackage[usenames]{color}

\newtheorem{Theorem}{Theorem}[section]
\newtheorem{Lemma}{Lemma}[section]
\newtheorem{Proposition}{Proposition}[section]
\newtheorem{Corollary}{Corollary}[section]

\theoremstyle{definition}

\theoremstyle{remark}
\newtheorem{Remark}{Remark}[section]

\numberwithin{equation}{section}

\renewcommand{\r}{\rho}

\def\i{\varepsilon}

\renewcommand{\u}{{\bf u}}

\newcommand{\R}{{\mathbb R}}
\newcommand{\Dv}{{\rm div}}
\newcommand{\Cu}{{\rm curl}}
\newcommand{\tr}{{\rm tr}}

\newcommand{\dl}{\delta}

\def\f{\frac}
\renewcommand{\O}{\Omega}
\def\ov{\overline}

\def\D{\Delta }

\def\hf1{^\f{1}{1-\xi^2}}

\def\be{\begin{equation}}
\def\en{\end{equation}}
\def\bs{\begin{split}}
\def\es{\end{split}}

\newcommand{\F}{{\mathtt F}}

\author{Xianpeng Hu and Fanghua Lin}
\address{Department of Mathematics, City University of Hong Kong, Hong Kong, PRC.} \email{xianpehu@cityu.edu.hk}
\address{Courant Institute of Mathematical Sciences, New York
University, New York, NY 10012.} \email{linf@cims.nyu.edu}

\title[Incompressible viscoelastic fluid]
{On the Cauchy problem for two dimensional incompressible viscoelastic 
flows}

\date{\today}

\begin{document}

\begin{abstract}
We study the large-data Cauchy problem for two dimensional Oldroyd model of incompressible 
viscoelastic fluids. We prove the global-in-time existence of the Leray-Hopf type weak 
solutions in the physical energy space. Our method relies on a new $\textit{a priori}$ 
estimate on the space-time norm in $L^{\f32}_{loc}$ of the Cauchy-Green strain tensor 
$\tau=\F\F^\top$, or equivalently the $L^3_{loc}$ norm of the Jacobian of the flow map $\F$. 
It allows us to rule out possible concentrations of the energy due to deformations associated 
with the flow maps. Following the general compactness arguments due to DiPerna and 
Lions (\cite{DL}, \cite{FNP}, \cite{PL}), and using the so-called \textit{effective viscous flux}, $\mathcal{G}$, 
which was introduced in our previous work \cite{HL}, we are able to control the possible 
oscillations of deformation gradients as well.
\end{abstract}

\maketitle

\section{Introduction}

We study the Cauchy problem for the Oldroyd model of an incompressible viscoelastic fluids 
(see \cite{MB, CD, SE, GS, JO1, JO, RHN, LFH}):
\begin{equation}\label{ae1}
\begin{cases}
\partial_t\u+\u\cdot\nabla\u-\mu\D\u+\nabla P=\Dv(\F\F^\top),\\
\Dv\u=0
\end{cases}
\end{equation}
for $(x,t)\in \R^d\times\R^+$ $(d=2,3)$, where $\u\in\R^d$ denotes the velocity of the fluid, $\F\in \mathcal{M}^d$ is the deformation gradient, $\mathcal{M}^d$ is a set of $d\times d$ matrices with $\det \F=1$, and $P$ is the pressure of the fluid, which is a Lagrangian multiplier due to the incompressibility condition $\Dv\u=0$. For a given velocity field $\u(x,t)\in\R^d$, one defines the flow map $x(t,X)$ by
\begin{equation}\label{ae2}
\begin{cases}
\f{dx}{dt}=\u(x,t),\\
x(0)=X.
\end{cases}
\end{equation}
>From \cite{DL}, the flow map would be uniquely defined when the vector field $\u(t,x)$ is in an 
appropriate Sobolev space, and $$\F(x,t)=\tilde{\F}(X,t)=\f{\partial x}{\partial X}(t,X)$$ 
is the Jacobian of the flow map $x(t,X)$, or oftenly it will be called the deformation 
gradient. It is well-known that the incompressibility condition, $\Dv\u=0$, is equivalent 
to $\det\F=1$ (see Proposition 1.4 in \cite{BM}). Moreover, $\F$ satisfies, see \cite{LINLIUZHANG, BM}
\begin{equation}\label{ae3}
\partial_t\F+\u\cdot\nabla\F=\nabla\u\F.
\end{equation}
It is then easy to deduce that \eqref{ae1} is equivalent to
\begin{equation}\label{e1}
\begin{cases}
\partial_t\u+\u\cdot\nabla\u-\mu\D\u+\nabla P=\Dv(\F\F^\top),\\
\partial_t\F+\u\cdot\nabla\F=\nabla\u\F,\\
\Dv\u=0,
\end{cases}
\end{equation}
for smooth solutions. Here $\mu>0$ is the fluid viscosity. Throughout this paper we shall 
assume $\mu=1$ as it would not effect our analysis.

Let $\F^\top$ be the transpose of the matrix $\F=(\F_1,\F_2)$ where $\F_1$ and $\F_2$ are columns of $\F$, and let $\Dv\F^\top$ be defined as $\partial_{x_i}\F_{ij}$, a direct computation yields
\begin{equation}\label{ae4}
\partial_t(\Dv\F^\top)+\u\cdot\nabla(\Dv\F^\top)=0,
\end{equation}
whenever $(\u,\F)$ is a smooth solution of \eqref{e1}. In particular
\begin{equation}\label{c}
%\begin{cases}
\Dv\F^\top(t)=0,\\
%\det\F(t)=1
%\end{cases}
\end{equation}
for all $t> 0$ if \eqref{c} is valid at $t=0$. With the identity \eqref{c} in hand, the second equation 
in \eqref{e1} can be written as
$$\partial_t\F_j+\u\cdot\nabla\F_j=\F_j\cdot\nabla\u$$ for $j=1,2$. Since $\Dv\F_j=0$ due to 
$\Dv\F^\top=0$, the equation for $\F_j$ can be further rewritten as
$$\partial_t\F_j+\Dv(\F_j\otimes\u-\u\otimes\F_j)=0,$$
where $(a\otimes b)_{ij}=a_ib_j.$
As a consequence, an integration of \eqref{ae3} implies that the zero-th order moment of $\F$ is conserved along the flow
\begin{equation}\label{ae13}
\f{d}{dt}\int_{\R^2}\Big(\F(x,t)-I\Big)dx=0.
\end{equation}

This paper studies the Cauchy problem for the system \eqref{e1} in two spatial dimensions with the 
initial condition:
\begin{equation}\label{IC}
(\u, \F)|_{t=0}=(\u_0, \F_0)(x) 
\end{equation}
such that 
\begin{equation}\label{IC1}
\det\F_0=1,\quad \Dv\F_0^\top=0,\quad\textrm{ and } \quad\Dv\u_0=0.
\end{equation} 
One can easily formulate the 
problem on a periodic box or a smooth bounded domain, see for example \cite{LZ}. We note
that for the Oldroyd-model \eqref{ae1}, $\F_0(x)$ may be assumed to be the identity matrix.
It may be however more convienent sometime to start with a general initial data as 
\eqref{IC}.

For classical solutions of \eqref{e1} and \eqref{IC}, with the aid of \eqref{ae13}, the 
basic energy identity valids (see also \cite{LINLIUZHANG, LZ}) :
\begin{equation}\label{ae5}
\begin{split}
\f12\f{d}{dt}\int_{\R^2}\Big[|\u|^2+|\F-I|^2\Big]dx=-\mu\int_{\R^2}|\nabla\u|^2dx.
\end{split}
\end{equation}
The difficulty to understand \eqref{e1}, which is a coupled system of parabolic-hyperbolic 
equations, is similar to that of the compressible Navier-Stokes equations. However, 
the equation for the nonnegative scalar function (density) $\r$ of the latter is now replaced by 
an equation for a matrix valued function $\F$, the deformation gradient of the flow map. The 
equation for $\F$ is the same one as that for the vorticity $\omega$ of the classical Euler 
equation. To study such an equation, it is not surprising that certain algebraic relations 
between unknowns involving $\F$ would be important (if not necessarily) in 
our analysis. Indeed, we shall explore the following identities which are 
true for smooth solutions:
\begin{equation}\label{ae6}
\Dv\F^\top=0,\quad \det\F=1,\quad\textrm{and}\quad \Cu \F^{-1}=0.
\end{equation}
The first two identities as we have described before are due to the incompressibility of the 
fluid. The last identity may be replaced by (an equivalent one) the Piola identity (see for 
example \cite{MB, SE, LEILIUZHOU}),
\begin{equation}\label{c1}
\F_{lk}\partial_{x_l}\F_{ij}(t)=\F_{lj}\partial_{x_l}\F_{ik}(t)
\end{equation}
for $i,j,k=1,\cdots, d$. The first two identities turn out to be  sufficient for us to 
derive the main \textit{a priori} estimates when $d=2$. For problems in three spatial 
dimensions, we believe the third identity would be crucial as in \cite{LZ, LEILIUZHOU}. But 
there 
are several additional difficult issues in 3D, and this explains 
why we just show the global existence of solutions in two dimensions.

By a quick examination on the system \eqref{e1} and by the (probably the only one) known 
\textit{a priori} estimate given by the energy law \eqref{ae5}, one finds the main difficulty in 
obtaining the global weak solutions (the case of two dimensions is not in exception) is the 
lack of compactness under the weak convergence of the nonlinear term $\F\F^\top$ in \eqref{e1}. The 
$L^1$ estimate (from \eqref{ae5}) on $\F\F^\top$ simply does not work well. This lack of 
the compactness or the higher integrability has been the major obstruction to a complete 
understanding of the Cauchy problem for 
\eqref{e1}. On the other hand, over last many years various authors have contributed to the 
study of \eqref{e1} and related problems, mostly for classical or strong solutions. Some of 
the earlier contributions may be found in the works \cite{CM, HWU, KS, LEILIUZHOU, 
LINLIUZHANG, PC4, PC, GS} and references therein.  
 
In particular, for \eqref{e1}-\eqref{IC}, authors in \cite{CZ, LEILIUZHOU, LINLIUZHANG, 
LZ} showed a global in time existence of  $H^2$ solutions to \eqref{e1}-\eqref{IC} 
whenever the initial data is a small $H^2$ perturbation around the equilibrium $(\u,\F)=(0, I)$, 
where $I$ is the identity matrix. 
The global existence of classical solutions in $H^s$ ($s\ge 8$) to \eqref{e1}-\eqref{IC} 
with $\mu=0$ was established in \cite{TS, TS1} via a vector field method since constraints 
\eqref{c} and \eqref{c1} may be used to show a dispersive structure for perturbations of 
$(\u,\F)$, which are sufficient when the spatial dimension is three. The case of two 
dimensions is critical and much hard, and the problem was settled recently in \cite{LEI, LS}. In \cite{PC1} authors proved global existence for small data with large gradients for 
Oldroyd-B, see also \cite{CM} for related discussions. Regularity for diffusive Oldroyd-B 
equations in dimension two for large data were obtained in the creeping flow regime 
(coupling with the stationary Stokes equations, rather than the Navier-Stokes equations) in 
\cite{PC2} and a more general version in \cite{PC3}. 

The global existence of strong solutions near the equilibrium for compressible models 
of \eqref{e1} has also been studied, see for examples \cite{HW} and the references 
therein.
Another interesting and related work is that by N. Masmoudi on the FENE dumbbell model 
\cite{MA}. He proved a global in time existence result for large initial data through a 
detailed analysis of the so-called defect measures associated with the approximate 
solutions.   

For \eqref{e1}, the global existence of mild solutions with discontinuous data near the 
equilibrium was studied recently by authors in \cite{HL}. 
In order to construct a solution, the initial data is chosen in a functional 
space which is very close to (be almost optimal) the natural energy space 
and the solution has no additional regularity assumptions. However the data 
are required to be 
small perturbations from the equilibrium. Nonetheless, the result in \cite{HL} may be viewed 
as the first step towarding the global physical solutions of the system 
\eqref{e1}-\eqref{IC} with the constraints \eqref{c}-\eqref{c1} in a suitable weak 
formulation. For the corotational Oldroyd-B model with a finite
relaxation time, the global existence of weak solutions with 
arbitrary initial data had been verified in \cite{LM}, see also \cite{BSS, BMA}.

In this paper, we are interested in the existence of global weak solutions of 
\eqref{e1}-\eqref{IC} in dimensions two with arbitrary large initial data in natural energy 
spaces. One starts with a sequence of smooth approximate solutions. We note that the 
construction of a smooth and suitable approximate solutions is by no mean straight 
forward, and it is done in the section 5 of this paper. The principle difficulty 
to deal with such a sequence of smooth (approximate) solutions would be to show the weak 
convergence of quadratic terms in the equations, and 
among them the convergence of $\F\F^\top$ at least in the sense of distributions.
It turns out that all other terms possess a div-curl structure. The term $\F\F^\top$ is 
the only term that does not seem to possess any additional structures.
To show the weak convergence of $\F\F^\top$ is not much different at this stage from 
showing the strong convergence of $\F\F^\top$. Here are two basic reasons:
\begin{itemize}
 \item The $L^2$ \textit{a priori} estimate on $\F-I$  ensures the 
convergence of $\F\F^\top$ merely in the space of Radon measures. The 
non-reflexive Banach 
space $L^1$ is not convenient to work with as bounded sequences are not necessarily
weakly precompact. In particular, the concentration phenomena may occur to prevent bounded 
sequences in this space from converging weakly to an integrable function.
\item Even if $\F\F^\top$ converges weakly in some functional space $L^p$ for $p\in[1,\infty)$ 
(that is, no concentrations), the weak limit of $\F\F^\top$ may not be given by the strain
tensor associated with the limiting flow map. The latter can be shown exist. In other words,
the sequence of deformation gradients may contain oscillations. 
\end{itemize}

The first difficulty may be resolved if one can show higher integrability of deformation
gradients $\F$.  By analyzing the stress tensor $\F\F^\top$, which is symmetric positive 
definite due to the incompressibility assumption $\det\F=1$, we succeed in
improving the integrability of $\F$ in $L_{loc}^{3}(\R^2\times\R^+)$. In fact, our arguments
yield a proof for $\F$ in $L_{loc}^{p}(\R^2\times\R^+)$ for all $2<p<4$. The detail of this
proof is carried out in the section 2 of the paper. The key idea is to estimate first the
symmetric tensor $\F\F^\top - tr(\F\F^\top)I/2$. This traceless tensor corresponds to 
exactly the so-called Hopf differentials associated with the flow maps. We thus believe the
method of estimating this traceless tensor $\F\F^\top - tr(\F\F^\top)I/2$ may be useful
for other applications of two dimensional problems.  
This higher integrability along with the fact that $\det\F=1$ leads to the $L^{3/2}$  
estimate on the tensors $\F\F^\top$, and the weak convergence of $\F\F^\top$ in the 
Lebesgue space $L_{loc}^{3/2}$. Another important consequence of this higher 
integrability is an improved estimate for the (total) hydrodynamic pressure $P$ (up to a 
constant) in $L_{loc}^{3/2}$.

To overcome the second difficulty relating to oscillations, authors introduced in 
\cite{HL} a combination between the velocity and the elastic stress, which is called the 
\textit{effective viscous flux}. It is defined as
$$\mathcal{G}=\nabla\u-(-\D)^{-1}\nabla\mathcal{P}\Dv(\F\F^\top),$$ where the operator 
$\mathcal{P}$ is the Leray-projection to the divergence free vector fields. One can easily 
check from the first equation of \eqref{e1} that
\begin{equation}\label{ae7}
\D\mathcal{G}=\nabla\mathcal{P}(\partial_t\u+\u\cdot\nabla\u)
\end{equation} 
using the incompressibility $\Dv\u=0$.
>From this and an \textit{a priori} estimates on $\mathcal{G}$, one deduces that 
$\mathcal{G}$ has better smoothness properties than either components of $\mathcal{G}$. In 
this paper, beside $\mathcal{G}$, we shall consider also different components of the first 
equation in \eqref{e1} and different versions of \textit{effective viscous flux}. They will
facilate us  to deal with the oscillation issue. In fact, as one will see in 
Section 3 below, the curl-free projection of the first equation in \eqref{e1} 
would be related to the ``total vorticity" and it is of particular importance in our 
analysis.

To describe the main result of this paper, we shall introduce a precise formulation of weak 
solutions. We say that the pair $(\u,\F)$ is a weak solution of \eqref{e1} with Cauchy data \eqref{IC} 
provided that $\F,\u,\nabla\u\in L^1_{loc}(\R^2\times\R^+)$ and for all
test functions $\beta, \psi\in\mathcal{D}(\R^2\times \R^+)$ with $\Dv\psi=0$ in $\mathcal{D}'(\R^2\times \R^+)$
\begin{equation}\label{e2}
\int_{\R^2}(\F_j)_0\beta(\cdot, 0)dx+\int_0^\infty\int_{\R^2}\Big[\F_j\beta_t+(\F_j\otimes\u-\u\otimes\F_j):\nabla\beta \Big]dxdt=0
\end{equation}
for $j=1,2$,
and
\begin{equation}\label{e3}
 \begin{split}
&\int_{\R^2}\u_0\psi(\cdot, 0)dx+\int_0^\infty\int_{\R^2}\Big[\u\cdot\partial_t\psi+(\u\otimes\u-\F\F^\top):\nabla\psi\Big] dxdt\\
&\quad=\int_0^\infty\int_{\R^2}\nabla\u:\nabla\psi dxdt.
 \end{split}
\end{equation}

Let us also introduce the notion of \textit{weak solutions with free energy} $(\u,\F)$ for
the Cauchy problem of \eqref{e1}-\eqref{IC} 
as follows:
\begin{itemize}
 \item The deformation graident $\F$ satisfies $$\F-I\in L^\infty([0,T]; L^2(\R^2)),\quad \F(x,0)=\F_0(x),$$
and the velocity satisfies
$$\nabla\u\in L^2([0,T];L^2(\R^2)),\quad \u(x,0)=\u_0(x).$$ Moreover $\u\otimes\u, \F\F^\top$ are locally integrable in $\R^2\times (0,T)$ and constraints \eqref{c}-\eqref{c1} hold true in $\mathcal{D}'(\R^2)$.
\item The system \eqref{e1} is satisfied in $\mathcal{D}'(\R^2\times(0,T))$; that is identities \eqref{e2} and \eqref{e3} hold.
\item The energy inequality
$$\f12\left(\|\u\|_{L^2}^2+\|\F-I\|_{L^2}^2\right)(t)+\int_0^t\|\nabla\u\|_{L^2}^2ds\le \f12\left(\|\u_0\|_{L^2}^2+\|\F_0-I\|_{L^2}^2\right)$$
holds true for almost all $t\in[0,T].$
\end{itemize}

Our main result may be stated as:
\begin{Theorem}\label{mt}
Assume that $\u_0\in L^2(\R^2)$ and $\F_0-I\in L^2(\R^2)$ with the constraint \eqref{IC1} in $\mathcal{D}'(\R^2)$. The Cauchy problem \eqref{e1}-\eqref{IC} with the constraint \eqref{c} admits a global-in-time \textit{weak solution  $(\u, \F)$ with free energy}. 
Moreover the solution $(\u,\F)$ satisfies
\begin{itemize}
\item The deformation gradient $\F$ is bounded in $L_{loc}^3(Q)$, and the hydrodynamic pressure $P $, up to a constant, is bounded in $L_{loc}^{3/2}(Q)$ with $Q=\R^2\times \R^+$;
\item The Piola identity \eqref{c1} and $\det\F=1$ hold true in $\mathcal{D}'(\R^2\times\R^+)$.
\end{itemize}
\end{Theorem}

If one ignores the viscoelastic effects, and hence $\F$ is dropped off, then the first 
equation of \eqref{e1} becomes the standard incompressible Navier-Stokes equations, and 
the weak solutions in Theorem \ref{mt} are simply the classical Leray-type weak solutions.
Though the initial data for $\F_0$ requires it to be close to $I$, due the structure of the 
equations, one can easily replace $I$ by any constant matrix of determinant one. On the
other hand, it is still an open question that whether or not the weak solutions described in
Theorem \ref{mt} would also be a weak solution to the original Oldroyd model \eqref{ae1}.
This is related to the regularity question, and we shall address it later in a forthcoming 
work. We also note, due to the local character of the weak compactness arguments, that the 
improved local higher integrabilities for $\F$ and $P$ are sufficient for our purpose. One 
believes that higher integrabilities may actually be valid globally. Since certain 
linear algebra facts do not seem to work out well when the spatial 
dimension is three, it remains as a fascinating open problem that whether 
or not the global existence result is also valid in higher dimensions.

Theorem \ref{mt} will be established by first constructing a sequence of smooth 
approximation solutions $(\u^\i, \F^\i)$, and then passing to the limit as $\i\rightarrow 
0$. This convergence analysis requires a great deal of technical and structural informations 
these regulairzed flows. The strategy of estimating the defect measure used in \cite{LM, MA} 
does not seem to be applicable to the system \eqref{e1} since the $L^2$ estimate of the 
stress tensor $\F\F^\top$ is still out of reach at present. Instead we employ the
approach of \cite{DL, PL} for general weak convergences. As we have discussed above,  
we use the higher integrability of $\F$ to eliminate the possibility of concentrations.
For possible oscillations, we also use re-normalizations of $\F$ and the convexity of the 
quantity $|\F|=\sqrt{\sum_{i,j=1}^2\F_{ij}^2}$. Indeed, we define (see \cite{FNP}) for any 
compact subset $\O\Subset \R^2$ the \textit{oscillation defect r-measure} of a sequence 
$|\F^\i|$ by $$osc[\F^\i\rightarrow \F](\O):=\sup_{k\ge 1}\limsup_{\i\rightarrow 
0}\|T_k(|\F^\i|)-T_k(|\F|)\|_{L^r(\O)}^r,$$
where $T_k(z)=kT(z/k)$ and $T(z)$ is a truncate function defined on positive reals such 
that it equals to $z$ for $z\le 1$ 
and takes a constant value $1$ for $z\ge 1$. We will show that the
\textit{oscillation defect $3$-measure} of $|\F^\i|$ is finite. The uniform control on the 
oscillation defect measure leads to a renormalized inequality for the deformation gradient, 
and this inequality is sufficient for our purpose of the strong convergence of the 
deformation gradient(see section 4 and 5 for details). In the process of achieving these, 
the effective viscous flux has played an important role, in particular for the 
\textit{oscillation defect 3-measure} of $\F$ to control the re-normalization of $\F$.

Let us end this discussion by making a comparision between the work of 
Masmoudi \cite{MA} on the FENE model and ours on the Oldroyd-model. 
We note that even if a damping term for $\tau$ is added to the \eqref{e1}, 
to show a global existence of weak solutions with the natural physical 
energy is still a very challenging problem. 
The FENE (\textit{Finite Extensible Nonlinear Elastic}) dumbbell model is described by
following equations:
\begin{equation*}
 \begin{cases}
\partial_t\u+\u\cdot\nabla\u-\D\u+\nabla P=\Dv\tau,\quad \Dv\u=0\\
\partial_t\psi+\u\cdot\nabla\psi=\Dv_{R}(-\nabla\u R\psi+\nabla\psi+\nabla\mathcal{U}\psi)\\
\tau_{ij}=\int_B R_i\nabla_j\mathcal{U}\psi(t,x, R)dR.
 \end{cases}
\end{equation*}
Note that formally if $\mathcal{U}= |R|^2$, it is easy to check that $\tau$ satisfies:
\begin{equation}\label{t}
\partial_t\tau+\u\cdot\nabla\tau+\tau=\nabla\u\tau+\tau(\nabla\u)^\top.
\end{equation} 
When the damping term $\tau$ is absent in \eqref{t} and $R$ allows to vary in the whole 
$\R^{2}$ the equation \eqref{t} is exactly the 
equation of the elastic stress $\tau=\F\F^\top$ in \eqref{e1} , see \eqref{20} in Section 2.
In \cite{MA}, the $L^2$ estimate of $\tau$ was obtained using diffusions 
of Fokker-Planck operators along with the Hardy type inequalities (see Corollary 3.6 in \cite{MA}). One 
notices that for the co-rotational model \cite{LM} such an estimate is 
automatic due to the maximum principle. Whether or not the $L^2$ estimate 
of the elastic stress $\tau=\F\F^\top$ is valid
for the Oldroyd-model considered here remains to be an interesting open question.

The rest of this paper is organized as follows. In Section 2, we present a proof for the 
higher integrability of the deformation gradient $\F$ and the pressure $P$ for smooth 
solutions of \eqref{ae1}. The stability of solutions is addressed in the sections 3 and 4,
again for smooth solutions. That is, we prove that weak limits of a sequence of smooth
solutions of \eqref{ae1} is a weak solution for both \eqref{ae1} and \eqref{e1}.
In Section 3, various versions of the \textit{effective viscous flux} are introduced, and 
they are used to handle oscillations of (approximate) solutions. Section 4 is devoted to the 
weak compactness issue. In Section 5 we introduce an approximate system and then finish the 
proof of Theorem \ref{mt}. An interesting point (beside the usual vanishing viscosity 
method) here is to replace $|\F|^{2}$ by a more nonlinear and convex elastic energy 
$W(\F)$, see \cite{LFH}. An application of the decomposition \eqref{200} is presented in 
Section 6, which provide an alternative approach to uniform estimate of $\F - I$ in 
\cite{HL}.

\bigskip\bigskip

\bigskip\bigskip

%%%%%%%%%%%%%%%%%%%%%%%%%%%%%%%%%%%%%%%%%%%%%%%%%%%%%%%%%%%%%%%%%%%%%%%%

\section{Higher Integrability of Deformation Gradients}
To prove the higher integrability of the deformation gradients $\F$, we shall assume 
first that the solutions $(u, \F)$ are sufficiently smooth.

To begin with, we introduce a decomposition of the symmetric tensor 
$\tau=\F\F^\top$ (see also \cite{MB,SE}). Denote 
$$2\Pi_1=|\F|^2=|\F^\top_1|^2+|\F^\top_2|^2=\tr(\F\F^\top),\quad 2\Pi_2=|\F^\top_1|^2-|\F^\top_2|^2,\quad \Pi_3=\F^\top_1\cdot\F^\top_2,$$ where $\F^\top_i$ is understood to be $(\F^\top)_i$.

In terms of $\Pi_i$, $1\le i\le 3$, the term $\Dv(\F\F^\top)$ involving the elastic stress 
can be written as
\begin{equation}\label{200}
\begin{split}
\Dv(\F\F^\top)&=\nabla\Pi_1+(\partial_1\Pi_2, -\partial_2\Pi_2)+(\partial_2\Pi_3, \partial_1\Pi_3)\\
&=\nabla\Pi_1+\Dv\left(\begin{array}{cc}\Pi_2 &\Pi_3\\
\Pi_3&-\Pi_2\end{array}\right).
\end{split}
\end{equation}

In view of the equation of $\F$, the equation for the symmetric tensor $\tau$ is given by
\begin{equation}\label{20}
 \partial_t \tau+\u\cdot\nabla \tau=\nabla\u \tau+\tau(\nabla\u)^\top,
\end{equation}
which coincides with the so-called upper convected derivative in literatures. Taking the 
trace of \eqref{20} yields an equation for $\tr\tau=2\Pi_1$:
\begin{equation*}
\partial_t\tr\tau+\u\cdot\nabla\tr\tau=\tr(\nabla\u\tau+\tau(\nabla\u)^\top),
\end{equation*}
which in turn implies that
\begin{equation}\label{20a}
 \partial_t(\tr\tau)^{-\f12}+\u\cdot\nabla(\tr\tau)^{-\f12}=-\f12(\tr\tau)^{-\f32}\tr(\nabla\u\tau+\tau(\nabla\u)^\top).
\end{equation}
By the polar decomposition, there exist an orthogonal matrix $R\in O(2)$ and $U\in \mathcal{M}$ such that
\begin{equation}\label{21}
\F=RU.
\end{equation}
Since $\tau$ is symmetric and positive definite and $\det \F=1$, there exist an orthogonal 
matrix $\mathcal{O}\in O(2)$ and a function $\lambda(t,x)$ with
$0<\lambda(t,x)\le 1$ such that
\begin{equation*}
\tau=\F\F^\top=\mathcal{O}^\top\left(\begin{array}{ccc}\lambda & 0\\0 &\lambda^{-1}\end{array}\right)\mathcal{O}.
\end{equation*}
As in \cite{FJ}, the orthogonal matrix $\mathcal{O}$ stands for the rotation, while the 
diagonal matrix, $\textrm{diag}\left(\lambda,\lambda^{-1}\right)$,  measures 
stretchings in two principle directions. The latter is also an indicator
of the anisotropies of the model \eqref{e1}. Moreover, since for any symmetric positive 
matrix M of size $d\times d$,
$$(\det M)^{\f1d}\le \f{1}{d}\tr M,$$
one deduces that $2\le\tr \tau$ in dimensions two and hence there holds
$$\tr (\tau-I)\ge 0.$$

%%%%%%%%%%%%%%%%%%%%%%%%%%%%%%%%%%%%%%%%%%%%
\subsection{Estimates of $\Pi_2$ and $\Pi_3$}

Thanks to \eqref{200}, the momentum equation now can be written as
\begin{equation}\label{24}
\partial_t\u+\u\cdot\nabla \u-\D\u+\nabla \hat{P}=\Dv\left(\begin{array}{cc}\Pi_2 &\Pi_3\\
\Pi_3&-\Pi_2\end{array}\right)
\end{equation}
with $\hat{P}$ being the total pressure $P-\Pi_1$. There are five observations from 
\eqref{24}:
\begin{itemize}
\item Taking curl of \eqref{24}, one has 
$$\partial_t\Cu\u+\Cu\Dv(\u\otimes 
\u)-\D\Cu\u=2\partial_{1}\partial_{2}\Pi_2+(\partial_2^2-\partial_1^2)\Pi_3,$$ and 
hence $(-\D)^{-1}\Big[2\partial_{1}\partial_{2}\Pi_2+(\partial_2^2-\partial_1^2)\Pi_3\Big]$ 
should be bounded in space-time $L^p$ for some $p>1$. 
\item Changing variables by
$y_1=\f{1}{\sqrt{2}}(x_1-x_2)$ and $y_2=\f{1}{\sqrt{2}}(x_1+x_2)$,\footnote{This is equivalent to a rotation of the coordinate system with an angle $\pi/4$ counterclockwise.}
one further deduces from the first observation above that 
$(-\D)^{-1}\Big[(\partial_1^2-\partial_2^2)\Pi_2+2\partial_{1}\partial_{2}\Pi_3\Big]$ is also bounded in $L^p$.
\item Taking the divergence of \eqref{24} yields
\begin{equation}\label{24b}
\D\hat{P}=-\Dv\Dv(\u\otimes\u)+(\partial_1^2-\partial_2^2)\Pi_2+2\partial_{1}\partial_{2}\Pi_3,
\end{equation}
by the incompressibility $\Dv\u=0$.
Thanks to the second observation, the regularity of elliptic equations guarantees that 
the total pressure $\hat{P}$ would be bounded in $L^p$.
\item Applying the operator $(\partial_{1}, -\partial_{2})$ to \eqref{24}, one obtains
\begin{equation*}
 \begin{split}
\D\Pi_2&=2\partial_t\partial_{1}\u_1+\partial_{1}\Dv(\u\u_1)-\partial_{2}\Dv(\u\u_2)-2\D\partial_{1}\u_1+(\partial_{1}^2-\partial_{2}^2)\hat{P},
 \end{split}
\end{equation*}
again by the incompressibility condition $\Dv\u=0.$
With the third item above at hand, the regularity of elliptic equations implies that 
$\Pi_2$ is again bounded in $L^p$.
\item Applying the operator $(\partial_2, \partial_1)$ to \eqref{24} yields
\begin{equation*}
 \begin{split}
\D\Pi_3&=\partial_t(\partial_{2}\u_1+\partial_1\u_2)+\partial_{2}\Dv(\u\u_1)+\partial_{1}\Dv(\u\u_2)-\D(\partial_{2}\u_1+\partial_1\u_2)+2\partial_{1}\partial_{2}\hat{P},
 \end{split}
\end{equation*}
and hence an estimate of $\Pi_3$ in $L^p$ follows as before.
\end{itemize}
The remainder of this subsection is devoted to the rigorous verifications of these 
formal observations. To this end, we shall denote the average-zero part of a function 
$f$ over the ball $B_a(0)$ by $[f]_a$, that is
\begin{equation}\label{a0}
[f]_a=f-\f{1}{\pi a^2}\int_{B_a(0)}f dx,
\end{equation}
for every ball $B_a(0)$ centred at $0$ with radius $a>0$.

We start with the proof of the first observation.
\begin{Lemma}\label{l1} 
For every ball $B_a(0)$ centred at $0$ with the radius $a>0$, there exists a constant $C(\i, a,T)$ such that
\begin{equation}\label{l1a}
\begin{split}
&\sum_{i=2}^3\left|\int_0^T\int_{B_a(0)}(-\D)^{-2}\partial_1\partial_2\Big[2\partial_{1}\partial_{2}\Pi_2+(\partial_2^2-\partial_1^2)\Pi_3\Big]\Big[\Pi_i(\tr\tau)^{-\f12}\phi \Big]_adxds\right|\\
&\quad+\sum_{i=2}^3\left|\int_0^T\int_{B_a(0)}(-\D)^{-2}(\partial_1^2-\partial_2^2)\Big[2\partial_{1}\partial_{2}\Pi_2+(\partial_2^2-\partial_1^2)\Pi_3\Big]\Big[\Pi_i(\tr\tau)^{-\f12}\phi\Big]_a dxds\right|\\
&\le C(\i, a, T,\phi)\Big(\|\u\|_{L^\infty(0,T; L^2(\R^2))}^3+\|\nabla\u\|_{L^2(\R^2\times(0,T))}^3+\|\F-I\|_{L^\infty(0,T; L^2(\R^2))}^{3}+1\Big)\\
&\quad+\i\|\phi(\tr\tau)^{\f12}\|_{L^3(B_a(0)\times(0,T))}^3
\end{split}
\end{equation}
for any $0<\i\le 1$, where $\phi$ is a radially symmetric smooth function with support in $B_a(0)$ and $0\le \phi\le 1$.
\end{Lemma}
\begin{proof}
The proof for the cases $i=2$ and $i=3$ are identical and we shall focus on $i=2$ for 
an illustration.

Applying the operator $(-\D)^{-1}\Cu$ to \eqref{24}, one has
\begin{equation*}
(-\D)^{-1}\Big[2\partial_{1}\partial_{2}\Pi_2+(\partial_2^2-\partial_1^2)\Pi_3\Big]=(-\D)^{-1}\partial_t\Cu\u+(-\D)^{-1}\Cu\Dv(\u\otimes\u)+\Cu\u.
\end{equation*}
Taking the inner product of the identity above with $\Big[\Pi_2(\tr\tau)^{-\f12}\phi\Big]_a$, one obtains
\begin{equation}\label{25}
\begin{split}
&\int_0^T\int_{B_a(0)}(-\D)^{-2}\partial_1\partial_2\Big[2\partial_{1}\partial_{2}\Pi_2+(\partial_2^2-\partial_1^2)\Pi_3\Big]\Big[\Pi_2(\tr\tau)^{-\f12}\phi\Big]_a dxds\\
&\quad=\int_0^T\int_{B_a(0)}\Big[(-\D)^{-2}\partial_1\partial_2\Cu\Dv(\u\otimes\u)\Big]\Big[\Pi_2(\tr\tau)^{-\f12}\phi\Big]_a dxds\\
&\qquad+\int_0^T\int_{B_a(0)}\Big[(-\D)^{-1}\partial_1\partial_2\Cu \u\Big]\Big[\Pi_2(\tr\tau)^{-\f12}\phi\Big]_a dxds\\
&\qquad+\int_0^T\int_{B_a(0)}\Big[(-\D)^{-2}\partial_1\partial_2\partial_t\Cu\u\Big]\Big[\Pi_2(\tr\tau)^{-\f12} \phi\Big]_a dxds\\
&\quad=\sum_{i=1}^3I_i.
\end{split}
\end{equation}

For $I_1$, since $\|[f]_a\|_{L^3(B_a(0))}\le C\|f\|_{L^3(B_a(0))}$, one gets easily that
\begin{equation*}
\begin{split}
|I_1|&=\left|\int_0^T\int_{B_a(0)}\Big[(-\D)^{-2}\partial_1\partial_2\Cu\Dv(\u\otimes \u)\Big]\Big[\Pi_2(\tr\tau)^{-\f12} \phi\Big]_a dxds\right|\\
&\le\|(-\D)^{-2}\partial_1\partial_2\Cu\Dv(\u\otimes \u)\|_{L^3(0,T; L^{\f32}(B_a(0)))}\Big\|\Big[\Pi_2(\tr\tau)^{-\f12} \phi\Big]_a\Big\|_{L^{\f32}(0,T; L^3(B_a(0)))}\\
&\le\|(-\D)^{-2}\partial_1\partial_2\Cu\Dv(\u\otimes \u)\|_{L^3(0,T; L^{\f32}(\R^2))}\|\phi\Pi_2(\tr\tau)^{-\f12}\|_{L^{\f32}(0,T; L^3(B_a(0)))}\\
&\le C\|\u\|_{L^6(0,T; L^3(\R^2))}^2\left\|\phi\Pi_2(\tr\tau)^{-\f12}\right\|_{L^{\f32}(0,T; L^3(B_a(0)))}\\
&\le C(T)\|\u\|_{L^6(0,T; L^3(\R^2))}^2\|\phi(\tr\tau)^{\f12}\|_{L^3(B_a(0)\times(0,T))}\\
&\le C(\i, T)\|\u\|_{L^6(0,T; L^3(\R^2))}^3+\i\|\phi(\tr\tau)^{\f12}\|_{L^3(B_a(0)\times(0,T))}^3\\
&\le C(\i,T)\Big(\|\u\|_{L^\infty(0,T; L^2(\R^2))}^3+\|\nabla\u\|_{L^2(\R^2\times(0,T))}^3\Big)+\i\|\phi(\tr\tau)^{\f12}\|_{L^3(B_a(0)\times(0,T))}^3.
\end{split}
\end{equation*}
Here in the fifth line we have used the fact $|\Pi_2|(\tr\tau)^{-\f12}\le 
C(\tr\tau)^{\f12}$, and in the last line we have applied Gagliardo-Nirenberg's inequality
$$\|\u\|_{L^6(0,T; L^3(\R^2))}\le \|\u\|_{L^\infty(0,T; L^2(\R^2))}^{\f23}\|\nabla \u\|_{L^2(\R^2\times(0,T))}^{\f13}.$$

For $I_2$, one observes that
\begin{equation*}
\begin{split}
|I_2|&\le \|\Cu\u\|_{L^2(B_a(0)\times(0,T))}\Big\|\Big[\phi\Pi_2(\tr\tau)^{-\f12}\Big]_a\Big\|_{L^{2}(B_a(0)\times(0,T))}\\
&\le \|\Cu\u\|_{L^2(B_a(0)\times(0,T))}\Big\|\Big[\phi\Pi_2(\tr\tau)^{-\f12}\Big]_a\Big\|_{L^{3}(B_a(0)\times(0,T))}\\
&\le \|\Cu\u\|_{L^2(B_a(0)\times(0,T))}\|\phi\Pi_2(\tr\tau)^{-\f12}\|_{L^{3}(B_a(0)\times(0,T))}\\
&\le C(a,T)\|\nabla\u\|_{L^2(\R^2\times(0,T))}\|\phi(\tr\tau)^{\f12}\|_{L^{3}(B_a(0)\times(0,T))}\\
&\le C(\i,a,T)\|\nabla\u\|^{\f32}_{L^2(\R^2\times(0,T))}+\i\|\phi(\tr\tau)^{\f12}\|_{L^{3}(B_a(0)\times(0,T))}^3.
\end{split}
\end{equation*}

For $I_3$, using integration by parts, one has
\begin{equation*}
 \begin{split}
I_3&=\int_0^T\int_{B_a(0)}\Big[(-\D)^{-2}\partial_1\partial_2\partial_t\Cu\u\Big]\Big[\phi\Pi_2(\tr\tau)^{-\f12}\Big]_a dxds\\
&=-\int_0^T\int_{B_a(0)}\Big[(-\D)^{-2}\partial_1\partial_2\Cu\u\Big]\partial_t\Big[\phi\Pi_2(\tr\tau)^{-\f12}\Big]_a dxds\\
&\quad+\int_{B_a(0)}\Big[(-\D)^{-2}\partial_1\partial_2\Cu\u\Big]\Big[\phi\Pi_2(\tr\tau)^{-\f12}\Big]_a dx\Big |_{s=0}^{s=T}.
 \end{split}
\end{equation*}
By the equation \eqref{20a} one obtains that
\begin{equation*}
\begin{split}
\partial_t\Big(\Pi_2(\tr\tau)^{-\f12}\Big)+\u\cdot\nabla\Big(\Pi_2(\tr\tau)^{-\f12}\Big)&=(\tr\tau)^{-\f12}(\F^\top\nabla\u_1\cdot\F^\top_1-\F^\top\nabla\u_2\cdot\F^\top_2)\\
&\quad-\f12\Pi_2(\tr\tau)^{-\f32}\tr(\nabla\u\tau+\tau(\nabla\u)^\top),
\end{split}
\end{equation*}
and hence
\begin{equation*}
\begin{split}
\partial_t\Big [\phi\Pi_2(\tr\tau)^{-\f12}\Big]_a+\Big[\phi\u\cdot\nabla\Big(\Pi_2(\tr\tau)^{-\f12}\Big)\Big]_a&=\Big[\phi(\tr\tau)^{-\f12}(\F^\top\nabla\u_1\cdot\F^\top_1-\F^\top\nabla\u_2\cdot\F^\top_2)\Big]_a\\
&\quad-\f12\Big[\phi\Pi_2(\tr\tau)^{-\f32}\tr(\nabla\u\tau+\tau(\nabla\u)^\top)\Big]_a.
\end{split}
\end{equation*}
The latter yields
\begin{equation}\label{251}
 \begin{split}
|I_3|&\le \left|\int_0^T\int_{B_a(0)}\Big[(-\D)^{-2}\partial_1\partial_2\Cu\u\Big]\Big[\phi\u\cdot\nabla\Big(\Pi_2(\tr\tau)^{-\f12}\Big)\Big]_adxds\right|\\
&\quad+\f12\left|\int_0^T\int_{B_a(0)}\Big[(-\D)^{-2}\partial_1\partial_2\Cu\u\Big]\Big[\phi\Pi_2(\tr\tau)^{-\f32}\tr\Big[\nabla\u\tau+\tau(\nabla\u)^\top\Big]\Big]_adxds\right|\\
&\quad+\left|\int_0^T\int_{B_a(0)}\Big[(-\D)^{-2}\partial_1\partial_2\Cu\u\Big]\Big[\phi(\tr\tau)^{-\f12}(\F^\top\nabla\u_1\cdot\F^\top_1-\F^\top\nabla\u_2\cdot\F^\top_2)\Big]_adxds\right|\\
&\quad+\left|\int_{B_a(0)}\Big[(-\D)^{-2}\partial_1\partial_2\Cu\u\Big]\Big[\phi\Pi_2(\tr\tau)^{-\f12} dx\Big]_adx\Big |_{s=0}^{s=T}\right|\\
&=\sum_{i=1}^4I_i.
 \end{split}
\end{equation}
We now estimate $I_{3_i}$ term by term. One first notices that
\begin{equation}\label{ia}
\int_{B_a(0)}[f]_agdx=\int_{B_a(0)}f[g]_a dx.
\end{equation}
Thus we can write $I_{3_1}$ as
$$I_{3_1}=\left|\int_0^T\int_{B_a(0)}\Big[(-\D)^{-2}\partial_1\partial_2\Cu\u\Big]_a\Big[\phi\u\cdot\nabla\Big(\Pi_2(\tr\tau)^{-\f12}\Big)\Big]dxds\right|,$$
which, by combining the incompressibility $\Dv\u=0$ and integrations by parts, 
leads to
\begin{equation*}
 \begin{split}
I_{3_1}&\le\left|\int_0^T\int_{B_a(0)}\Big[(-\D)^{-2}\partial_1\partial_2\nabla\Cu\u\Big]\cdot\u \phi\Pi_2(\tr\tau)^{-\f12}dxds\right|\\
&\quad+\left|\int_0^T\int_{B_a(0)}\Big[(-\D)^{-2}\partial_1\partial_2\Cu\u\Big]_a\nabla\phi\cdot\u \Pi_2(\tr\tau)^{-\f12}dxds\right|\\
&\le C(T)\|\u\|^2_{L^6(0,T; L^3(\R^2))}\|\phi(\tr\tau)^{\f12}\|_{L^3(B_a(0)\times(0,T))}\\
&\quad+C(T,\phi)\|(\tr\tau)^{\f12}\|_{L^\infty(0,T; L^2(B_a(0)))}\Big\|\Big[(-\D)^{-2}\partial_1\partial_2\Cu\u\Big]_a\Big\|_{L^4((0,T)\times B_a(0))}\|\u\|_{L^4((0,T)\times\R^2)}\\
&\le C(\i,T)\Big(\|\u\|_{L^\infty(0,T; L^2(\R^2))}^3+\|\nabla\u\|_{L^2(\R^2\times(0,T))}^3\Big)+\i\|\phi(\tr\tau)^{\f12}\|_{L^3(B_a(0)\times(0,T))}^3\\
&\quad+C(a,T, \phi)\Big(\|\F-I\|_{L^\infty(0,T; L^2(\R^2))}+1\Big)\|\nabla\u\|^{\f12}_{L^2(0,T;L^2(\R^2))}\|\u\|^{\f32}_{L^\infty(0,T; L^2(\R^2))}\\
&\le C(\i,a,T,\phi)\Big(\|\u\|_{L^\infty(0,T; L^2(\R^2))}^3+\|\nabla\u\|_{L^2(\R^2\times(0,T))}^3+\|\F-I\|_{L^\infty(0,T; L^2(\R^2))}^{3}+1\Big)\\
&\quad+\i\|\phi(\tr\tau)^{\f12}\|_{L^3(B_a(0)\times(0,T))}^3. 
\end{split}
\end{equation*}
Where we have applied the fact $\tr\tau=|\F|^2$ and that
\begin{equation}\label{te}
\|(\tr\tau)^{\f12}\|_{L^2(B_a(0))}\le C(a)(\|\F-I\|_{L^2}+1).
\end{equation} 
Moreover, in the third inequality of the above estimate for $I_{3_1}$, the following 
Sobolev inequality 
$\|[f]_a\|_{L^\infty(B_a(0))}\le C a^{\f12}\|\nabla f\|_{L^4(B_a(0))}$, was also used 
so that one has:
\begin{equation*}
\begin{split}
\Big\|\Big[(-\D)^{-2}\partial_1\partial_2\Cu\u\Big]_a\Big\|_{L^4((0,T)\times B_a(0))}&\le Ca^{\f12}\Big\|\Big[(-\D)^{-2}\partial_1\partial_2\Cu\u\Big]_a\Big\|_{L^4(0,T; L^\infty(B_a(0)))}\\
&\le Ca \|\nabla(-\D)^{-2}\partial_1\partial_2\Cu\u\|_{L^4((0,T)\times B_a(0))}\\
&\le Ca \|\u\|_{L^4((0,T)\times L^2(\R^2))}\\
&\le Ca \|\u\|^{\f12}_{L^\infty(0,T; L^2(\R^2))}\|\nabla\u\|_{L^2((0,T)\times\R^2)}^{\f12}.
\end{split}
\end{equation*}
For $I_{3_2}$, by \eqref{ia}, $|(\tr\tau)^{-\f12}\Pi_2|\le C(\tr\tau)^{\f12}$, and 
$|\tau|\le C\tr\tau$, one deduces
\begin{equation*}
\begin{split}
I_{3_2}&=\f12\left|\int_0^T\int_{B_a(0)}\Big[(-\D)^{-2}\partial_1\partial_2\Cu\u\Big]\Big[\phi\Pi_2(\tr\tau)^{-\f32}\tr\Big[\nabla\u\tau+\tau(\nabla\u)^\top\Big]\Big]_adxds\right|\\
&=\f12\left|\int_0^T\int_{B_a(0)}\Big[(-\D)^{-2}\partial_1\partial_2\Cu\u\Big]_a\phi\Pi_2(\tr\tau)^{-\f32}\tr\Big[\nabla\u\tau+\tau(\nabla\u)^\top\Big]dxds\right|\\
&\le C(T)\Big\|\Big[(-\D)^{-2}\partial_1\partial_2\Cu\u\Big]_a\Big\|_{L^6((0,T)\times B_a(0))}\|\nabla\u\|_{L^2(\R^2\times(0,T))}\|\phi(\tr\tau)^{\f12}\|_{L^3(B_a(0)\times(0,T))}\\
&\le C(a,T)\|\u\|^{\f23}_{L^\infty(0,T; L^2(\R^2))}\|\nabla\u\|^{\f43}_{L^2(\R^2\times(0,T))}\|\phi(\tr\tau)^{\f12}\|_{L^3(B_a(0)\times(0,T))}\\
&\le C(\i,a,T)\Big(\|\u\|_{L^\infty(0,T; L^2(\R^2))}^3+\|\nabla\u\|_{L^2(\R^2\times(0,T))}^3\Big)+\i\|\phi(\tr\tau)^{\f12}\|_{L^3(B_a(0)\times(0,T))}^3,
\end{split}
\end{equation*}
where in the fourth line above we have used the following estimate, which is again due 
to the Sobolev inequality $\|[f]_a\|_{L^\infty(B_a(0))}\le C a^{\f13}\|\nabla 
f\|_{L^3(B_a(0))},$
\begin{equation*}
\begin{split}
\Big\|\Big[(-\D)^{-2}\partial_1\partial_2\Cu\u\Big]_a\Big\|_{L^6((0,T)\times B_a(0))}&\le Ca^{\f13}\Big\|\Big[(-\D)^{-2}\partial_1\partial_2\Cu\u\Big]_a\Big\|_{L^6(0,T; L^\infty(B_a(0)))}\\
&\le Ca^{\f23} \|\nabla(-\D)^{-2}\partial_1\partial_2\Cu\u\|_{L^6(0,T; L^3(B_a(0)))}\\
&\le Ca^{\f23} \|\u\|_{L^6(0,T;L^3(\R^2))}\\
&\le Ca^{\f23} \|\u\|^{\f23}_{L^\infty(0,T; L^2(\R^2))}\|\nabla\u\|_{L^2((0,T)\times\R^2)}^{\f13}.
\end{split}
\end{equation*}
For $I_{3_3}$, one proceeds similarly as that for $I_{3_2}$. One can estimate $I_{3_3}$
as follows:
\begin{equation*}
I_{3_3}\le C(\i, a,T)\Big(\|\u\|_{L^\infty(0,T; L^2(\R^2))}^3+\|\nabla\u\|_{L^2(\R^2\times(0,T))}^3\Big)+\i\|\phi(\tr\tau)^{\f12}\|_{L^3(B_a(0)\times(0,T))}^3.
\end{equation*}
For $I_{3_4}$, by applying the identity \eqref{ia}, the Poincare's inequality, and the 
inequality \eqref{te}, one gets
\begin{equation*}
\begin{split}
|I_{3_4}|&=\left|\int_{B_a(0)}\Big[(-\D)^{-2}\partial_1\partial_2\Cu\u\Big]_a\phi\Pi_2(\tr\tau)^{-\f12} dx\Big |_{s=0}^{s=T}\right|\\
&\le 2\sup_{s\ge 0}\left|\int_{B_a(0)}\Big[(-\D)^{-2}\partial_1\partial_2\Cu\u\Big]_a\phi\Pi_2(\tr\tau)^{-\f12} dx\right|\\
&\le 2\sup_{s\ge 0}\left(\Big\|\Big[(-\D)^{-2}\partial_1\partial_2\Cu\u\Big]_a\Big\|_{L^2}(s)\|\phi(\tr\tau)^{\f12}\|_{L^2(B_a(0))}(s)\right)\\
&\le C(a)\|\u\|_{L^\infty(0,T; L^2(\R^2))}\Big(\|\F-I\|_{L^\infty(0,T; 
L^2(\R^2))}+1\Big).
\end{split}
\end{equation*}

Combining all these estimates above, we finally arrive at 
\begin{equation*}
\begin{split}
&\left|\int_0^T\int_{B_a(0)}(-\D)^{-2}\partial_1\partial_2\Big[2\partial_{1}\partial_{2}\Pi_2+(\partial_2^2-\partial_1^2)\Pi_3\Big]\Big[\Pi_2(\tr\tau)^{-\f12}\phi \Big]_adxds\right|\\
&\quad\le C(\i, a, T,\phi)\Big(\|\u\|_{L^\infty(0,T; L^2(\R^2))}^3+\|\nabla\u\|_{L^2(\R^2\times(0,T))}^3+\|\F-I\|_{L^\infty(0,T; L^2(\R^2))}^{3}+1\Big)\\
&\qquad+\i\|\phi(\tr\tau)^{\f12}\|_{L^3(B_a(0)\times(0,T))}^{3}.
\end{split}
\end{equation*}

The other parts of \eqref{l1a} can be handled in the same manner, and the conclusion of the 
Lemma follows.
\end{proof}

Next we turn the attention to an estimate for $\hat{P}$.
\begin{Lemma}\label{l2}
For every ball $B_a(0)$ centred at $0$ with the radius $a>0$, there exists a constant $C(\i, a,T)$ such that
\begin{equation}\label{l2a}
\begin{split}
&\sum_{i=2}^3\left|\int_0^T\int_{B_a(0)}(-\D)^{-1}\partial_1\partial_2\hat{P}\Big[\Pi_i(\tr\tau)^{-\f12}\phi \Big]_adxds\right|\\
&\quad+\sum_{i=2}^3\left|\int_0^T\int_{B_a(0)}(-\D)^{-1}(\partial_1^2-\partial_2^2)\hat{P}\Big[\Pi_i(\tr\tau)^{-\f12}\phi\Big]_a dxds\right|\\
&\le C(\i, a, T,\phi)\Big(\|\u\|_{L^\infty(0,T; L^2(\R^2))}^3+\|\nabla\u\|_{L^2(\R^2\times(0,T))}^3+\|\F-I\|_{L^\infty(0,T; L^2(\R^2))}^{3}+1\Big)\\
&\quad+ \i\|\phi(\tr\tau)^{\f12}\|_{L^3(B_a(0)\times(0,T))}^3
\end{split}
\end{equation}
for any $0<\i\le 1$, where $\phi$ is a radially symmetric smooth function with support in $B_a(0)$ and $0\le \phi\le 1$.
\end{Lemma}
\begin{proof}
By a change of variables with
$y_1=\f{1}{\sqrt{2}}(x_1-x_2)$ and $y_2=\f{1}{\sqrt{2}}(x_1+x_2)$, one can easily deduce that
\begin{equation*}
\begin{split}
&\left|\int_0^T\int_{B_a(0)}(-\D)^{-2}\partial_1\partial_2\Big[(\partial_1^2-\partial_2^2)\Pi_2+2\partial_{1}\partial_{2}\Pi_3\Big]\Big[\Pi_i(\tr\tau)^{-\f12}\phi\Big]_a dxds\right|\\
&\quad=
\left|\int_0^T\int_{B_a(0)}(-\D)^{-2}(\partial_1^2-\partial_2^2)\Big[2\partial_1\partial_2\Pi_2+(\partial_{2}^2-\partial_{1}^2)\Pi_3\Big]\Big[\Pi_i(\tr\tau)^{-\f12}\phi\Big]_a dxds\right|,
\end{split}
\end{equation*}
and hence via Lemma \ref{l1},  one gets
\begin{equation}\label{24c}
\begin{split}
&\sum_{i=2}^3\left|\int_0^T\int_{B_a(0)}(-\D)^{-2}\partial_1\partial_2\Big[(\partial_1^2-\partial_2^2)\Pi_2+2\partial_{1}\partial_{2}\Pi_3\Big]\Big[\Pi_i(\tr\tau)^{-\f12}\phi\Big]_a dxds\right|\\
&\quad+\sum_{i=2}^3\left|\int_0^T\int_{B_a(0)}(-\D)^{-2}(\partial_1^2-\partial_2^2)\Big[(\partial_1^2-\partial_2^2)\Pi_2+2\partial_{1}\partial_{2}\Pi_3\Big]\Big[\Pi_i(\tr\tau)^{-\f12}\phi\Big]_a dxds\right|\\
&\quad\le C(\i, a, T,\phi)\Big(\|\u\|_{L^\infty(0,T; L^2(\R^2))}^3+\|\nabla\u\|_{L^2(\R^2\times(0,T))}^3+\|\F-I\|_{L^\infty(0,T; L^2(\R^2))}^{3}+1\Big)\\
&\qquad+\i\|\phi(\tr\tau)^{\f12}\|_{L^3(B_a(0)\times(0,T))}^3.
\end{split}
\end{equation}

Taking the divergence of \eqref{24} yields
$$\D\hat{P}=-\Dv\Dv(\u\otimes\u)+(\partial_1^2-\partial_2^2)\Pi_2+2\partial_{1}\partial_{2}\Pi_3,$$
and consequently, one has
\begin{equation*}
\begin{split}
&\int_0^T\int_{B_a(0)}(-\D)^{-1}\partial_1\partial_2\hat{P}\Big[\Pi_i(\tr\tau)^{-\f12}\phi\Big]_a dxds\\
&=\int_0^T\int_{B_a(0)}(-\D)^{-2}\partial_1\partial_2\Dv\Dv(\u\otimes\u)\Big[\Pi_i(\tr\tau)^{-\f12}\phi\Big]_a dxds\\
&\quad-\int_0^T\int_{B_a(0)}(-\D)^{-2}\partial_1\partial_2\Big[(\partial_1^2-\partial_2^2)\Pi_2+2\partial_{1}\partial_{2}\Pi_3\Big]\Big[\Pi_i(\tr\tau)^{-\f12}\phi\Big]_a dxds.
\end{split}
\end{equation*}
Using \eqref{24c} and a similar argument as that in the proof of Lemma \ref{l1}, one obtains, 
for a constant $C(\i, a, T)$,  the desired estimate that
\begin{equation*}
\begin{split}
&\sum_{i=2}^3\left|\int_0^T\int_{B_a(0)}(-\D)^{-1}\partial_1\partial_2\hat{P}\Big[\Pi_i(\tr\tau)^{-\f12}\phi\Big]_adxds\right|\le\i\|\phi(\tr\tau)^{\f12}\|_{L^3(B_a(0)\times(0,T))}^3\\
&\quad+ C(\i,a,T,\phi)\Big(\|\u\|_{L^\infty(0,T; 
L^2(\R^2))}^3+\|\nabla\u\|_{L^2(\R^2\times(0,T))}^3+\|\F-I\|_{L^\infty(0,T; L^2(\R^2))}^{3}+1\Big).
\end{split}
\end{equation*}

The estimates of the other parts in \eqref{l2a} follow in the same way.
\end{proof}

With the help of Lemma \ref{l2}, we can estimate further the quantities $\Pi_2$ and $\Pi_3$.
\begin{Lemma}\label{l3}
For every ball $B_a(0)$ centred at $0$ with the radius $a>0$, there exists a constant $C(\i, a,T)$ such that
\begin{equation*}
 \begin{split}
&\int_0^T\int_{B_a(0)}\phi\Pi_2^2(\tr\tau)^{-\f12}dxds+\int_0^T\int_{B_a(0)}\phi\Pi^2_3(\tr\tau)^{-\f12}dxds\\
&\quad\le C(\i, a, T,\phi)\Big(\|\u\|_{L^\infty(0,T; L^2(\R^2))}^3+\|\nabla\u\|_{L^2(\R^2\times(0,T))}^3+\|\F-I\|_{L^\infty(0,T; L^2(\R^2))}^{3}+1\Big)\\
&\qquad+\i\|\phi(\tr\tau)^{\f12}\|_{L^3(B_a(0)\times(0,T))}^3
 \end{split}
\end{equation*}
for any $0<\i\le 1$, where $\phi$ is a radially symmetric smooth function with support in $B_a(0)$ and $0\le \phi\le 1$.
\end{Lemma}

\begin{proof}
 Applying the operator $(\partial_{1}, -\partial_{2})$ to \eqref{24}, we arrive at
\begin{equation}\label{24d}
 \begin{split}
\D\Pi_2&=2\partial_t\partial_{1}\u_1+\partial_{1}\Dv(\u\u_1)-\partial_{2}\Dv(\u\u_2)-2\D\partial_{1}\u_1+(\partial_{1}^2-\partial_{2}^2)\hat{P}.
 \end{split}
\end{equation}
Taking the inner product of \eqref{24d} with $\Pi_2(\tr\tau)^{-\f12}$, one obtains
\begin{equation}\label{24e}
 \begin{split}
&\int_0^T\int_{B_a(0)} \Pi_2\Big[\phi\Pi_2(\tr\tau)^{-\f12}\Big]_a dxds\\
&\quad=-2\int_0^T\int_{B_a(0)} (-\D)^{-1}\partial_t\partial_{1}\u_1\Big[\Pi_2 (\tr\tau)^{-\f12}\phi\Big]_a dxds\\
&\qquad-\int_0^T\int_{B_a(0)} (-\D)^{-1}\Big[\partial_{1}\Dv(\u\u_1)-\partial_{2}\Dv(\u\u_2) \Big]\Big[\Pi_2(\tr\tau)^{-\f12}\phi\Big]_a dxds\\
&\qquad+2\int_0^T\int_{B_a(0)} \partial_{1}\u_1 \Big[\Pi_2(\tr\tau)^{-\f12}\phi\Big]_a dxds\\
&\qquad-\int_0^T\int_{B_a(0)} (-\D)^{-1}(\partial_{1}^2-\partial_{2}^2)\hat{P}\Big[\Pi_2(\tr\tau)^{-\f12}\phi\Big]_a dxds.
 \end{split}
\end{equation}
With the help of the conclusion and its proof in the Lemma \ref{l2}, one can estimate other 
terms in \eqref{24e} in a similar fashion as that for \eqref{25} to conclude that there exists 
a constant $C(\i,a,T, \phi)$ such that
\begin{equation*}
\begin{split}
\int_0^T\int_{B_a(0)} \Pi_2\Big[\phi\Pi_2(\tr\tau)^{-\f12}\Big]_a dxds
&\le C(\i,a,T,\phi)\Big(\|\u\|_{L^\infty(0,T; L^2(\O))}^3+\|\nabla\u\|_{L^2(\O\times(0,T))}^3\\
&\quad+\|\F-I\|_{L^\infty(0,T; 
L^2(\R^2))}^{3}+1\Big)+\i\|\phi(\tr\tau)^{\f12}\|_{L^3(\O\times(0,T))}^3.
\end{split}
\end{equation*}
The latter leads to the following estimates:
\begin{equation}\label{24f}
 \begin{split}
&\int_0^T\int_{B_a(0)}\phi\Pi_2^2(\tr\tau)^{-\f12}dxds\\
&\quad\le \int_0^T\int_{B_a(0)} \Pi_2\Big[\phi\Pi_2(\tr\tau)^{-\f12}\Big]_a dxds+\f{1}{\pi a^2}\int_0^T\int_{B_a(0)}|\Pi_2|dx \int_{B_a(0)}\phi|\Pi_2|(\tr\tau)^{-\f12}dxds\\
&\quad\le  \int_0^T\int_{B_a(0)} \Pi_2\Big[\phi\Pi_2(\tr\tau)^{-\f12}\Big]_a dxds+C(a,T)\|\F\|_{L^\infty(0,T; L^2(\R^2))}^3\\
&\quad\le C(\i, a, T,\phi)\Big(\|\u\|_{L^\infty(0,T; L^2(\R^2))}^3+\|\nabla\u\|_{L^2(\R^2\times(0,T))}^3+\|\F-I\|_{L^\infty(0,T; L^2(\R^2))}^{3}+1\Big)\\
&\qquad+\i\|\phi(\tr\tau)^{\f12}\|_{L^3(B_a(0)\times(0,T))}^3.
 \end{split}
\end{equation}

On the other hand, if one applies the operator $(\partial_{2}, \partial_{1})$ to \eqref{24}, 
one would obtain
\begin{equation}\label{24d1}
 \begin{split}
\D\Pi_3&=\partial_t(\partial_{2}\u_1+\partial_1\u_2)+\partial_{2}\Dv(\u\u_1)+\partial_{1}\Dv(\u\u_2)-2\D(\partial_{2}\u_1+\partial_1\u_2)+2\partial_{1}\partial_{2}\hat{P}.
 \end{split}
\end{equation}
Again, one takes the inner product of \eqref{24d1} with $\Pi_3(\tr\tau)^{-\f12}$ to conclude
\begin{equation}\label{24e1}
 \begin{split}
&\int_0^T\int_{B_a(0)}\Pi_3\Big[\phi\Pi_3(\tr\tau)^{-\f12}\Big]_a dxds\\
&\quad=-2\int_0^T\int_{B_a(0)} (-\D)^{-1}\partial_t(\partial_{2}\u_1+\partial_1\u_2) \Big[\Pi_3(\tr\tau)^{-\f12} \phi\Big]_a dxds\\
&\qquad-\int_0^T\int_{B_a(0)} (-\D)^{-1}\Big[\partial_{2}\Dv(\u\u_1)+\partial_{1}\Dv(\u\u_2) \Big]\Big[\Pi_3(\tr\tau)^{-\f12}\phi\Big]_a dxds\\
&\qquad+2\int_0^T\int_{B_a(0)} (\partial_{2}\u_1+\partial_1\u_2) \Big[\Pi_3(\tr\tau)^{-\f12}\phi\Big]_a dxds\\
&\qquad-2\int_0^T\int_{B_a(0)} (-\D)^{-1}\partial_{1}\partial_{2}\hat{P}\Big[\Pi_3(\tr\tau)^{-\f12}\phi\Big]_a dxds.
 \end{split}
\end{equation}
Following a similar argument as that in \eqref{24f}, one gets, for a constant 
$C(\i,a,T,\phi)$, that
\begin{equation}\label{24e1a}
\begin{split}
&\int_0^T\int_{B_a(0)} \Pi_3\Big[\phi\Pi_3(\tr\tau)^{-\f12}\Big]_a dxds\\
&\quad\le C(\i,a,T,\phi)\Big(\|\u\|_{L^\infty(0,T; L^2(\O))}^3+\|\nabla\u\|_{L^2(\O\times(0,T))}^3+\|\F-I\|_{L^\infty(0,T; L^2(\R^2))}^{3}+1\Big)\\
&\qquad+\i\|\phi(\tr\tau)^{\f12}\|_{L^3(\O\times(0,T))}^3.
\end{split}
\end{equation}
It is the desired estimate.
\end{proof}

\subsection{Improved estimates of $\tau$ and $P$}

We start with the following technical lemma.

\begin{Lemma}\label{l4a}
Let $\F\in\mathcal{M}^{2\times 2}$ with $\F_1$ and $\F_2$ being its first and second column respectively. There exists a universal positive constant $C$ such that
\begin{equation*}
|\F|^2\le C\left(\Big||\F_1|^2-|\F_2|^2\Big|+|\F_1\cdot\F_2|+|\det\F|\right).
\end{equation*}
\end{Lemma}
\begin{proof}
Without loss of generality, we assume that $|\F_1|\ge |\F_2|$. We consider two cases.

\texttt{Case 1: $|\F_1|\ge 2|\F_2|$.} In this situation, one has $|\F_2|^2\le \f{1}{4}|\F_1|^2$, and hence
\begin{equation*}
|\F_1|^2-|\F_2|^2\ge \f{3}{4}|\F_1|^2.
\end{equation*}
Therefore there follows
\begin{equation*}
\begin{split}
|\F|^2=|\F_1|^2+|\F_2|^2\le \f54|\F_1|^2\le \f53\Big(|\F_1|^2-|\F_2|^2\Big),
\end{split}
\end{equation*}
as one desired.

\texttt{Case 2: $|\F_1|\le 2|\F_2|$.} In this situation, $|\F_1|$ and $|\F_2|$ are comparable; that is
$$|\F_2|\le |\F_1|\le 2|\F_2|.$$
Let $\theta$ be the angle between $\F_1$ and $\F_2$. There holds 
$$\F_1\cdot\F_2=|\F_1||\F_2|\cos\theta. $$
We consider two subcases.

\textit{Subcase 1: $|\cos\theta|\ge \f12$.} In this situation, one has
$$|\F_1\cdot\F_2|=|\F_1||\F_2||\cos\theta|\ge \f12|\F_1||\F_2|,$$
and hence
\begin{equation*}
\begin{split}
|\F|^2=|\F_1|^2+|\F_2|^2\le 2|\F_1|^2\le 4|\F_1||\F_2|\le 8|\F_1\cdot\F_2|
\end{split}
\end{equation*}
as desired.

\textit{Subcase 2: $|\cos\theta|\le \f12$.} For any vector $v=(v_1, v_2)\in\R^2$, we set 
$\dot{v}=(v_2,-v_1)$. Note that $v$ is perpendicular to $\dot{v}$ and $|v|=|\dot{v}|$. We decompose $\F_1$ in the directions of $\F_2$ and $\dot{\F_2}$ as
$$\F_1=\f{\dot{\F_2}}{|\dot{\F_2}|}|\F_1|\sin\theta+\f{\F_2}{|\F_2|}|\F_1|\cos\theta,$$
and hence
$$\dot{\F_1}=-\f{\F_2}{|\dot{\F_2}|}|\F_1|\sin\theta+\f{\dot{\F_2}}{|\F_2|}|\F_1|\cos\theta.$$
Thus there follows
\begin{equation*}
\begin{split}
-\det\F=\dot{\F_1}\cdot\F_2=-\f{|\F_2|^2}{|\dot{\F_2}|}|\F_1|\sin\theta=-|\F_1||\F_2|\sin\theta.
\end{split}
\end{equation*}
Since $|\cos\theta|\le\f12$, the identity above yields
$$|\F_1||\F_2|=\f{|\det\F|}{|\sin\theta|}\le \f{2}{\sqrt{3}}|\det\F|.$$
This further implies that
\begin{equation*}
\begin{split}
|\F|^2=|\F_1|^2+|\F_2|^2\le 3|\F_1||\F_2|\le 2\sqrt{3}|\det\F|,
\end{split}
\end{equation*}
which is desired estimate.
\end{proof}

With the aid of Lemma \ref{l4a}, the improved integrability of $\tr\tau$, $P$ can be achieved 
as in the following proposition.
\begin{Proposition}\label{pp}
For every ball $B_a(0)$ centred at $0$ with the radius $a>0$, there exists a constant $C(a,T)$ such that
\begin{equation*}
 \begin{split}
&\|\tr\tau\|_{L^{\f32}(B_{a/2}(0)\times(0,T))}+\|P_{a/2}\|_{L^{\f32}(B_{a/2}(0)\times(0,T))}\\
&\quad\le C(a, T)\Big(\|\u\|^2_{L^\infty(0,T; L^2(\R^2))}+\|\nabla\u\|^2_{L^2(\R^2\times(0,T))}+\|\F-I\|^2_{L^\infty(0,T; L^2(\R^2))}+1\Big).
 \end{split}
\end{equation*}
\end{Proposition}
\begin{proof}
In view of Lemma \ref{l4a} with $\det\F=1$ and the Young's inequality, one has
\begin{equation*}
\begin{split}
(\tr\tau)^{\f32}=2^{\f32}\Pi_1^{\f32}&\le C\Big(|\Pi_2|\Pi_1^{\f12}+|\Pi_3|\Pi_1^{\f12}+\Pi_1^{\f12}\Big)\\
&\le C\Big(\Pi_2^2\Pi_1^{-\f12}+\Pi_3^2\Pi_1^{-\f12}+1\Big)+\f12\Pi_1^{\f32},
\end{split}
\end{equation*}
and therefore there is a universal constant $C>0$ such that
\begin{equation*}
\begin{split}
\Pi_1^{\f32}\le C\Big(\Pi_2^2\Pi_1^{-\f12}+\Pi_3^2\Pi_1^{-\f12}+1\Big).
\end{split}
\end{equation*}
The above inequality along with Lemma \ref{l3}, imply that
\begin{equation*}
\begin{split}
\|\phi^{\f23}\tr\tau\|_{L^{\f32}(B_a(0)\times(0,T))}^{\f32}&\le C\Big(\int_{B_a(0)\times(0,T)}\phi\Pi_2^2(\tr\tau)^{-\f12}dxds+\int_{B_a(0)\times(0,T)}\phi\Pi_3^2(\tr\tau)^{-\f12}dxds\\
&\quad+\int_{B_a(0)\times(0,T)}\phi dxds\Big)\\
&\le C(\i,a,T,\phi)\Big(\|\u\|_{L^\infty(0,T; L^2(\R^2))}^3+\|\nabla\u\|_{L^2(\R^2\times(0,T))}^3\\
&\quad+\|\	F-I\|_{L^\infty(0,T; L^2(\R^2))}^{3}+1\Big)+\i\|\phi^2\tr\tau\|_{L^{\f32}(B_a(0)\times(0,T))}^{\f32}\\
&\le C(\i,a,T,\phi)\Big(\|\u\|_{L^\infty(0,T; L^2(\R^2))}^3+\|\nabla\u\|_{L^2(\R^2\times(0,T))}^3\\
&\quad+\|\F-I\|_{L^\infty(0,T; 
L^2(\R^2))}^{3}+1\Big)+\i\|\phi^{\f23}\tr\tau\|_{L^{\f32}(B_a(0)\times(0,T))}^{\f32}.
\end{split}
\end{equation*}
Here we notice that $\phi^2\le \phi^{\f23}$ as $0\le \phi\le 1$.
Letting $\i\le \f12$, the estimate above leads to 
\begin{equation*}
\begin{split}
\|\phi^{\f23}\tr\tau\|_{L^{\f32}(B_a(0)\times(0,T))}
&\le C(a,T,\phi)\Big(\|\u\|^2_{L^\infty(0,T; L^2(\R^2))}+\|\nabla\u\|^2_{L^2(\R^2\times(0,T))}\\
&\quad+\|\F-I\|_{L^\infty(0,T; L^2(\R^2))}^2+1\Big),
\end{split}
\end{equation*}
which is the desired estimate for $\tr\tau$ by simply setting $\phi=1$ on $B_{a/2}(0)$.

Next, by taking the divergence of the momentum equation in \eqref{e1}, one gets 
$$\D P=\Dv\Dv(\tau-\u\otimes\u).$$
Hence a direct duality argument would give the desired conclusion that 
\begin{equation*}
\begin{split}
\|P_{a/2}\|_{L^{\f32}(B_{a/2}(0)\times(0,T))}&\le \|\tr\tau\|_{L^{\f32}(B_{a/2}(0)\times(0,T))}+\|\u\u\|_{L^{\f32}(B_{a/2}(0)\times(0,T))}\\
&\le \|\tr\tau\|_{L^{\f32}(B_{a/2}(0)\times(0,T))}+\|\u\|^2_{L^{3}(B_{a/2}(0)\times(0,T))}\\
&\le C(a,T)\Big(\|\u\|^2_{L^\infty(0,T; L^2(\R^2))}+\|\nabla\u\|^2_{L^2(\R^2\times(0,T))}\\
&\quad+\|\F-I\|_{L^\infty(0,T; L^2(\R^2))}^2+1\Big).
\end{split}
\end{equation*}
\end{proof}

\begin{Remark} 
It is not hard to check that the $L^{\f32}_{loc}$ estimates for $P_{a/2}$ and $\tau=\F\F^\top$ 
could be further improved to estimates $L^p_{loc}$ for any $1<p<2$ with 
the same arguments. Unfortunately the critical value $p=2$ does not seem to be approachable by 
the technique presented in this section. It may due to the hyperbolic 
nature of the equation for 
$\F$ and the presence of the stretching term $\nabla\u\F$ in this equation. 
But for our construction of global weak solutions, the $L_{loc}^{\f32}$ estimate is 
sufficient.
\end{Remark}

\bigskip\bigskip

%%%%%%%%%%%%%%%%%%%%%%%%%%%%%%%%%%%%%%%%%%%%%%%%%%%%%%%%%%%%%%%%%%%%%%%%%%%%%%%%%%%%%%%%%%

\section{Effective Viscous Flux and Oscillations}

We start with a sequence of solutions $(\u^\i,\F^\i)_{\{\i>0\}}$ of \eqref{e1}
\begin{equation}\label{e21}
 \begin{cases}
\partial_t\u^\i+\u^\i\cdot\nabla\u^\i-\D\u^\i+\nabla P^\i=\Dv(\F^\i(\F^\i)^\top)\\
\partial_t\F^\i+\u^\i\cdot\nabla\F^\i=\nabla\u^\i\F^\i\\
(\u^\i,\F^\i)|_{t=0}=(\u_0,\F_0),\quad \Dv\u^\i=0.
 \end{cases}
\end{equation}
where $\i>0$ is a small parameter. Moreover the sequence of
solutions $(\u^\i,\F^\i)$ of \eqref{e21} satisfies the energy inequality
\begin{equation}\label{ae14}
\begin{split}
\f12\Big(\|\u^\i\|_{L^2}^2+\|\F^\i-I\|_{L^2}^2\Big)(t)+\int_0^t\|\nabla\u^\i\|_{L^2}^2ds\le \f12\Big(\|\u_0\|_{L^2}^2+\|\F_0-I\|_{L^2}^2\Big).
\end{split}
\end{equation}
In this section and the next section, with the improved estimates for $(\u^\i, \F^\i, P^\i)$ uniformly in $\i$ in Section 2, we focus on the weak stability issue for a sequence of weak solutions to the system \eqref{e21} as the parameter $\i$ tends to zero. 

Based on this energy bounds above, up to subsequences, we can assume that for an 
arbitrary $T>0$
$$\u^\i\rightarrow \u\quad \textrm{weak$^*$ in}\quad L^\infty(0,T; L^2(\R^2))\cap L^2(0,T; H^1(\R^2))$$
and
$$\F^\i\rightarrow \F\quad\textrm{weak$^*$ in} \quad L^\infty(0,T; L^2_{loc}(\R^2)).$$
Then by a routine argument, one can get (see for example \cite{T})
$$\u^\i\cdot\nabla\u^\i\rightarrow \u\cdot\nabla\u\quad\textrm{in}\quad \mathcal{D}'(\R^+\times\R^2)$$
and
$$\F^\i\otimes\u^\i-\u^\i\otimes\F^\i\rightarrow \F\otimes\u-\u\otimes\F\quad\textrm{in}\quad \mathcal{D}'(\R^+\times\R^2).$$
Taking the limit as $\i\rightarrow 0$ in the momentum equation
of \eqref{e21}, one has
\begin{equation}\label{130}
\partial_t\u+\u\cdot\nabla\u-\D\u+\nabla
P=\Dv(\overline{\F\F^\top})\quad\textrm{in}\quad
\mathcal{D}'(\R^+\times\R^2),
\end{equation}
where the notation $\overline{f}$ means the weak limit in $L^1$ of
$\{f^\i\}_{\{\i>0\}}$. Moreover, one can easily deduce that
\begin{equation*}
\Dv(\ov{\F\F^\top})=\nabla\ov{\Pi_1}+(\partial_1\ov{\Pi_2}, -\partial_2\ov{\Pi_2})+(\partial_2\ov{\Pi_3}, \partial_1\ov{\Pi_3}),
\end{equation*}
and also that
\begin{equation}\label{ro1}
\partial_t\F+\u\cdot\nabla\F=\nabla\u\F\quad\textrm{in}\quad \mathcal{D}'((0,T)\times\R^2).
\end{equation}
Moreover the following also holds:
\begin{equation}\label{130a}
\partial_t\ov{\psi(\F)}+\u\cdot\nabla\ov{\psi(\F)}=\ov{\nabla\u\F\nabla_{\F}\psi(\F)}\quad\textrm{in}\quad \mathcal{D}'((0,T)\times\R^2)
\end{equation}
for all $\psi\in C^1(\mathcal{M})$.

\subsection{Effective viscous fluxes}

In \cite{HL}, the authors introduced the so-called \textit{effective viscous flux}, 
$$\mathcal{G}=\nabla\u-(-\D)^{-1}\nabla\mathcal{P}\Dv(\F\F^\top),$$
which has played an important role in controlling the propagation of the deformation 
gradient. However, the improved regularity and weighted $H^1$ estimates on 
$\mathcal{G}$ when the initial data are small perturbations from the equilibrium (as  
established in \cite{HL}) can not be expected here. Instead, we shall aim at  
a weak continuity property of $\mathcal{G}$ described in the following Lemma.

\begin{Lemma}\label{l61}
For solutions $(\u^\i, \F^\i)$ of \eqref{e21}, there holds true
\begin{equation}\label{137}
 \begin{split}
&\lim_{\i\rightarrow 0}\int_0^T\int_\O \Big[\nabla\u^\i-(-\D)^{-1}\nabla\mathcal{P}\Dv(\F^\i(\F^\i)^\top)\Big]:\psi(\F^\i) dxdt\\
&\quad=\int_0^T\int_\O\Big[\nabla\u-(-\D)^{-1}\nabla\mathcal{P}\Dv(\ov{\F\F^\top})\Big]:\ov{\psi(\F)}dxdt,
 \end{split}
\end{equation}
where $\psi\in C^1(\mathcal{M},\R^{2\times 2})$ satisfies the condition
\begin{equation}\label{132a}
\f{d\psi(\F)}{d\F}=0\quad\textrm{for all}\quad |\F|\ge M,
\end{equation}
for some large constant $M$ .
\end{Lemma}
\begin{proof}
We apply the operator $(-\D)^{-1}\nabla\mathcal{P}$ to the equation \eqref{e21} and 
use the divergence free property $\Dv\u^\i=0$ to find that
\begin{equation}\label{138}
\begin{split}
(-\D)^{-1}\partial_t\nabla\u^\i+(-\D)^{-1}\nabla\mathcal{P}\Dv(\u^\i\otimes\u^\i)+\nabla\u^\i=(-\D)^{-1}\nabla\mathcal{P}\Dv(\F^\i(\F^\i)^\top).
\end{split}
\end{equation}
Multipling the equation \eqref{138} by $\psi(\F^\i)$ and integrating, one obtains
\begin{subequations}\label{139}
\begin{align}
&\int_0^T\int_\O\Big[\nabla\u^\i-(-\D)^{-1}\nabla\mathcal{P}\Dv(\F^\i(\F^\i)^\top)\Big]:\psi(\F^\i)dxdt\nonumber\\
&\quad=-\int_0^T\int_\O(-\D)^{-1}\nabla\mathcal{P}\Dv(\u^\i\otimes\u^\i):\psi(\F^\i)dxdt\label{139a}\\
&\qquad-\int_0^T\int_\O \partial_t(-\D)^{-1}\nabla\u^\i:\psi(\F^\i) dxdt\label{139b}.
\end{align}
\end{subequations}
Integration by parts in \eqref{139b} leads to
\begin{subequations}\label{139b1}
\begin{align}
\eqref{139b}&=\int_0^T\int_\O (-\D)^{-1}\nabla\u^\i:\partial_t\psi(\F^\i)dxdt-\int_\O(-\D)^{-1}\nabla\u^\i(s):\psi(\F^\i)(s)dx\Big|_{s=0}^{s=T}\nonumber\\
&=-\int_0^T\int_\O (-\D)^{-1}\nabla\u^\i:\Big[\u^\i\cdot\nabla \psi(\F^\i)\Big]dxdt
\label{139b1a}\\
&\quad+\int_0^T\int_\O (-\D)^{-1}\nabla\u^\i:\Big[\nabla\u^\i\F^\i\nabla_{\F}\psi(\F^\i)\Big] dxdt\label{139b1b}\\
&\quad-\int_\O(-\D)^{-1}\nabla\u^\i(s):\psi(\F^\i)(s)dx\Big|_{s=0}^{s=T}.\label{139b1c}
\end{align}
\end{subequations}

Applying the same arguments this time to the equation \eqref{130} 
instead of \eqref{e21}), we obtain the following (in the place of \eqref{138}):
\begin{equation*}
\begin{split}
(-\D)^{-1}\partial_t\nabla\u+(-\D)^{-1}\nabla\mathcal{P}\Dv(\u\otimes\u)+\nabla\u=(-\D)^{-1}\nabla\mathcal{P}\Dv(\ov{\F\F^\top}).
\end{split}
\end{equation*}
Again, we multiply the above equation by the quantity 
$\ov{\psi(\F)}$ to obtain
\begin{subequations}\label{140}
\begin{align}
&\int_0^T\int_\O\Big[\nabla\u-(-\D)^{-1}\nabla\mathcal{P}\Dv(\ov{\F\F^\top})\Big]:\ov{\psi(\F)}dxdt\nonumber\\
&\quad=-\int_0^T\int_\O(-\D)^{-1}\nabla\mathcal{P}\Dv(\u\otimes\u):\ov{\psi(\F)}dxdt\label{140a}\\
&\qquad-\int_0^T\int_\O \partial_t(-\D)^{-1}\nabla\u:\ov{\psi(\F)} dxdt\label{140b}.
\end{align}
\end{subequations}
Using the equation \eqref{130a} along with an integration by parts in time yields
\begin{subequations}\label{140b1}
\begin{align}
\eqref{140b}&=\int_0^T\int_\O (-\D)^{-1}\nabla\u:\partial_t\ov{\psi(\F)} dxdt-\int_\O(-\D)^{-1}\nabla\u(s):\ov{\psi(\F)}(s)dx\Big|_{s=0}^{s=T}\nonumber\\
&=-\int_0^T\int_\O(-\D)^{-1}\nabla\u:\Big[\u\cdot\nabla\ov{\psi(\F)}\Big]dxdt\label{140b1a}\\
&\quad+\int_0^T\int_\O (-\D)^{-1}\nabla\u:\Big[\ov{\nabla\u\F\nabla_{\F}\psi(\F)}\Big]dxdt\label{140b1b}\\
&\quad-\int_\O(-\D)^{-1}\nabla\u(s):\ov{\psi(\F)}(s)dx\Big|_{s=0}^{s=T}\label{140b1c}.
\end{align}
\end{subequations}
We are ready now to verify various convergences.

\texttt{Convergence of \eqref{139a} and \eqref{139b1c}.} Since $\Dv\u^\i=0$, 
and $\Dv(\u^\i\otimes\u^\i)=\u^\i\cdot\nabla\u^\i$ belongs to the Hardy space 
$\mathcal{H}^1$ (see \cite{LIONS}), one has therefore
\begin{equation*}
(-\D)^{-1}\nabla\mathcal{P}\Dv(\u^\i\otimes\u^\i)\rightarrow (-\D)^{-1}\nabla\mathcal{P}\Dv(\u\otimes\u)\quad\textrm{weakly in}\quad L^2(0,T; W^{1,1}(\R^2)).
\end{equation*}
By Lemma 5.1 in \cite{PL}, and by the weak convergence of $\psi(\F^\i)$ to 
$\ov{\psi(\F)}$ in $L^\infty(\R^2\times(0,T))$, one may conclude that
$$\eqref{139a}\rightarrow \eqref{140a}\quad \textrm{as}\quad\i\rightarrow 0.$$
In the same manner, one can show
$$\eqref{139b1c}\rightarrow \eqref{140b1c}\quad\textrm{as}\quad\i\rightarrow 0.$$

\texttt{Convergence of \eqref{139b1a}.}  The basic energy inequality would imply that 
\begin{equation}\label{139b1a0}
(-\D)^{-1}\nabla\u^\i\rightarrow (-\D)^{-1}\nabla\u\quad\textrm{weakly$^*$ in}\quad L^2(0,T; H^2(\R^2))\cap L^\infty(0,T; H^1(\R^2)),
\end{equation}
and the incompressibility $\Dv\u^\i=0$ leads to
\begin{equation}\label{139b1a2}
\begin{split}
\u^\i\cdot\nabla \psi(\F^\i)=\Dv(\u^\i\psi(\F^\i))\rightarrow \Dv(\u\ov{\psi(\F)})=\u\cdot\nabla\ov{\psi(\F)}\quad\textrm{weakly in}\quad L^2(0,T;H^{-1}(\R^2)).
\end{split}
\end{equation}
Thus it follows from \eqref{139b1a0} and \eqref{139b1a2} that
$$\eqref{139b1a}\rightarrow \eqref{140b1a}\quad \textrm{as}\quad\i\rightarrow 0.$$

\texttt{Convergence of \eqref{139b1b}.} The $L^3_{loc}(\R^2\times\R^+)$ bound on 
$\F^\i$ implies
$$\nabla\u^\i\F^\i\nabla_{\F}\psi(\F^\i)\rightarrow 
\ov{\nabla\u\F\nabla_{\F}\psi(\F)}\quad\textrm{weakly in}\quad L^{\f65}(0,T; L_{loc}^{\f65}(\R^2)).$$
One therefore can deduce from \eqref{139b1a0} that
\begin{equation*}
\begin{split}
\eqref{139b1b}\rightarrow \eqref{140b1b}\quad \textrm{as}\quad \i\rightarrow 0.
\end{split}
\end{equation*}

In view of all these convergences, the desired identity \eqref{137} follows from 
\eqref{139} and \eqref{140}.
\end{proof}

\begin{Remark}
In establishing these convergences in the previous Lemma, either the 
classical Aubin-Lions lemma or the Lemma 5.1 in 
\cite{PL} has been applied implicitly in the proof. We note that some modifications 
may actually be needed in the proof of Lemma \ref{l61}. But it would have 
added some tedious details that are not essential in arguments. We rather prefer
to leave it to careful readers.
\end{Remark}

Next, we shall consider the curl-free part of the projection of the first equation of 
\eqref{e2}. It will be useful to control possible oscillations of $\F$.
We  start with the quantity $$(-\D)^{-1}\Dv\Big(\nabla P-\Dv(\F\F^\top)\Big).$$ 
Using the incompressibility $\Dv\u=0$, we shall prove the following
\begin{Corollary}\label{l31}
For solutions $(\u^\i, \F^\i)$ of \eqref{e21}, there holds true
\begin{equation}\label{132}
 \begin{split}
&\lim_{\i\rightarrow 0}\int_0^T\int_\O (-\D)^{-1}\Dv\Big[\nabla P^\i-\Dv(\F^\i(\F^\i)^\top)\Big]\phi(\F^\i) dxdt\\
&\quad=\int_0^T\int_\O(-\D)^{-1}\Dv\Big[\nabla P-\Dv(\ov{\F\F^\top})\Big]\ov{\phi(\F)}dxdt,
 \end{split}
\end{equation}
where $\phi\in C^1(\mathcal{M},\R)$ satisfies the condition \eqref{132a}.
\end{Corollary}
\begin{proof}
Applying the operator $(-\D)^{-1}\Dv$ to \eqref{e1} yields
$$(-\D)^{-1}\Dv\Dv(\u^\i\otimes\u^\i)+(-\D)^{-1}\Dv\Big[\nabla P^\i-\Dv(\F^\i(\F^\i)^\top)\Big]=0.$$
Multipling this equation by $\phi(\F^\i)$ and perform an integration, one has
\begin{equation}\label{133}
\begin{split}
&\int_0^T\int_\O(-\D)^{-1}\Dv\Big[\nabla P^\i-\Dv(\F^\i(\F^\i)^\top)\Big] \phi(\F^\i) dxdt\\
&\quad=-\int_0^T\int_\O(-\D)^{-1}\Dv\Dv(\u^\i\otimes\u^\i)\phi(\F^\i) dxdt.
\end{split}
\end{equation}

Similarly, by applying the operator $(-\D)^{-1}\Dv$ to \eqref{130}, one arrives at
$$(-\D)^{-1}\Dv\Dv(\u\otimes\u)+(-\D)^{-1}\Dv\Big[\nabla P-\Dv(\ov{\F\F^\top})\Big]=0.$$
Multipling this equation by $\ov{\phi(\F)}$, one would obtain 
\begin{equation}\label{134}
\begin{split}
&\int_0^T\int_\O(-\D)^{-1}\Dv\Big[\nabla P-\Dv(\ov{\F\F^\top})\Big] \ov{\phi(\F)} dxdt\\
&\quad=-\int_0^T\int_\O(-\D)^{-1}\Dv\Dv(\u\otimes\u)\ov{\phi(\F)} dxdt.
\end{split}
\end{equation}

Since
$$\phi(\F^\i)\rightarrow \ov{\phi(\F)}\quad \textrm{weakly in}\quad C_{weak}(0,T; L^3(\R^2)),$$
the weak convergence of $(-\D)^{-1}\Dv\Dv(\u^\i\otimes\u^\i)$ to 
$(-\D)^{-1}\Dv\Dv(\u\otimes\u)$ in $L^2(0,T; W^{2,1}(\R^2))$ implies
\begin{equation*}
\begin{split}
&-\int_0^T\int_\O(-\D)^{-1}\Dv\Dv(\u^\i\otimes\u^\i)\phi(\F^\i) dxdt\\
&\quad\rightarrow-\int_0^T\int_\O(-\D)^{-1}\Dv\Dv(\u\otimes\u)\ov{\phi(\F)} dxdt.
\end{split}
\end{equation*}
Combining this statement with \eqref{133} and \eqref{134}, one gets the desired conclusion 
\eqref{132}.
\end{proof}

\begin{Remark}
In terms of the decomposition \eqref{200}, the identity \eqref{132} may be rewritten as
\begin{equation}\label{135}
\begin{split}
&\lim_{\i\rightarrow 0}\int_0^T\int_\O \Big[P_a^\i-\Pi_1^\i+(-\D)^{-1}(\partial_1^2-\partial_2^2)\Pi_2^\i+2(-\D)^{-1}\partial_1\partial_2\Pi_3^\i\Big]\phi(\F^\i) dxdt\\
&\quad=\int_0^T\int_\O \Big[P_a-\ov{\Pi_1}+(-\D)^{-1}(\partial_1^2-\partial_2^2)\ov{\Pi_2}+2(-\D)^{-1}\partial_1\partial_2\ov{\Pi_3}\Big]\ov{\phi(\F)} dxdt,
\end{split}
\end{equation}
where $P_a$ is defined as \eqref{a0} with $\O\Subset B_a(0)$.
\end{Remark}

\subsection{The oscillation defect measure} The goal of this subsection is to control the oscillation of the sequence $\{|\F^\i|\}_{\{\i>0\}}$.
\begin{Lemma}\label{l4}
There exists a constant $c$ independent of $k$ such that
\begin{equation}\label{l43}
\limsup_{\i\rightarrow 0}\|T_k(|\F^\i|)-T_k(\ov{|\F|})\|_{L^3(\O\times[0,T])}\le c
\end{equation}
for any compact subset $\O\Subset\R^2$.
\end{Lemma}
\begin{proof}
>From the definition of $T_k$, one finds, for all $a,b>0$, that
$$0\le |T_k(a)-T_k(b)|^3\le (a^2-b^2)(T_k(a)-T_k(b)).$$
The convexity of the function $a\in (0,\infty)\mapsto a^2$ along with the concavity of 
$T_k(a)$ imply the following statements:
$$\ov{|\F|^2}\ge \ov{|\F|}^2\quad\textrm{and}\quad \ov{T_k(|\F|)}\le 
T_k(\ov{|\F|})\quad\textrm{a.e. in}\quad Q=\R^2\times\R^+.$$
Thus there holds
\begin{equation}\label{l41}
 \begin{split}
&\limsup_{\i\rightarrow 0}\int_0^T\int_\O \Big|T_k(|\F^\i|)-T_k(\ov{|\F|})\Big|^3 dxdt\\
&\quad\le\lim_{\i\rightarrow 0}\int_0^T\int_\O \Big(|\F^\i|^2-\ov{|\F|}^2\Big)\Big(T_k(|\F^\i|)-T_k(\ov{|\F|})\Big)dxdt\\
&\quad\le\lim_{\i\rightarrow 0}\int_0^T\int_\O \Big(|\F^\i|^2-\ov{|\F|}^2\Big)\Big(T_k(|\F^\i|)-T_k(\ov{|\F|})\Big)dxdt\\
&\qquad-\int_0^T\int_\O \Big(\ov{|\F|^2}-\ov{|\F|}^2\Big)\Big(\ov{T_k(|\F|)}-T_k(\ov{|\F|})\Big)dxdt\\
&\quad=
\lim_{\i\rightarrow 0}\int_0^T\int_\O |\F^\i|^2T_k(|\F^\i|)-\ov{|\F|^2}\,\ov{T_k(|\F|)}dxdt.
 \end{split}
\end{equation}

On the other hand, since $2\Pi_1^\i=|\F^\i|^2$ and $2\ov{\Pi_1}=\ov{|\F|^2}$, one deduces from \eqref{135} that
\begin{equation}\label{l42}
 \begin{split}
&\lim_{\i\rightarrow 0}\int_0^T\int_\O |\F^\i|^2T_k(|\F^\i|)-\ov{|\F|^2}\,\ov{T_k(|\F|)}dxdt\\
&\quad=2\lim_{\i\rightarrow 0}\int_0^T\int_\O \Big[P_a^\i+(-\D)^{-1}(\partial_1^2-\partial_2^2)\Pi_2^\i+2(-\D)^{-1}\partial_1\partial_2\Pi_3^\i\Big]T_k(|\F^\i|)dxdt\\
&\qquad-2\int_0^T\int_\O \Big[P_a+(-\D)^{-1}(\partial_1^2-\partial_2^2)\ov{\Pi_2}+2(-\D)^{-1}\partial_1\partial_2\ov{\Pi_3}\Big]\ov{T_k(|\F|)} dxdt\\
&\quad=2\lim_{\i\rightarrow 0}\int_0^T\int_\O \Big[P_a^\i+(-\D)^{-1}(\partial_1^2-\partial_2^2)\Pi_2^\i+2(-\D)^{-1}\partial_1\partial_2\Pi_3^\i\Big] \Big[T_k(|\F^\i|)-T_k(\ov{|\F|})\Big]dxdt\\
&\qquad-2\int_0^T\int_\O \Big[P_a+(-\D)^{-1}(\partial_1^2-\partial_2^2)\ov{\Pi_2}+2(-\D)^{-1}\partial_1\partial_2\ov{\Pi_3}\Big]\Big[\ov{T_k(|\F|)}-T_k(\ov{|\F|})\Big] dxdt\\
&\quad\le 2\limsup_{\i\rightarrow 0}\|P_a^\i+(-\D)^{-1}(\partial_1^2-\partial_2^2)\Pi_2^\i+2(-\D)^{-1}\partial_1\partial_2\Pi_3^\i\|_{L^{\f32}(\O\times(0,T))}\\
&\qquad\times\|T_k(|\F^\i|)-T_k(\ov{|\F|})\|_{L^3(\O\times(0,T))}\\
&\quad\le C(\O,T)\limsup_{\i\rightarrow 0}\|T_k(|\F^\i|)-T_k(\ov{|\F|})\|_{L^3(\O\times(0,T))}.
 \end{split}
\end{equation}

Combining \eqref{l41} and \eqref{l42} together, one has the desired estimate \eqref{l43}.

\end{proof}

\bigskip\bigskip

%%%%%%%%%%%%%%%%%%%%%%%%%%%%%%%%%%%%%%%%%%%%%%%%%%%%%%%%%%%%%%%%%%%%%%%%%%%%%%%%%%%%%%%%%%

\section{Renormalization and Weak Stability}

The aim of this section is to show the weak stability of the approximate solutions 
$(\u^\i,\F^\i)$ of \eqref{e21}. This procedure involves a renormalization argument with 
the aid of the convexity, see \cite{DL, FNP, PL}.

\subsection{Renormalization of $\F$} The Renormalization of the equation for $\F$ is 
obviously the key part. Roughly speaking one can renormalize the equation of 
$\F$ by simply multiplying both its sides by a function $\nabla_\F b(\F)$, and obtain a
transport equation for a new function $b(\F)$, see for instance \cite{DL, FNP, PL} and 
references therein. The justification of this approach is a delicate process because one 
needs to make sure that the multiplication between the equation of $\F$ and 
$\nabla_\F b(\F)$ is meaningful.
\begin{Lemma}\label{l5}
\begin{itemize}
\item 
The limit function $(\u,\F)$ satisfies the equation in the sense of renormalized solutions, i.e.,
\begin{equation}\label{l51}
 \partial_tb(\F)+\u\cdot\nabla b(\F)=\nabla\u\F: \nabla_{\F}b(\F)
\end{equation}
holds in $\mathcal{D}'((0,T)\times \R^2)$ for any $b\in C^2(\mathcal{M};\R)$ such that
$$\nabla_{\F}b(\F)=0\quad\textrm{for all \quad$\F\in\mathcal{M}$\quad with\quad $|\F|\ge M$}$$ where the constant $M$ may vary for different functions $b$ and the $2\times 2$ matrix $\nabla_{\F}b(\F)$ is
given by $$\nabla_{\F} b(\F)=\f{\partial b(\F)}{\partial\F}.$$
\item If $(\u^\dl, \F^\dl)_{\dl>0}$ is a sequence of renormalized solutions, then its limit $(\u, \F)$ as $\dl\rightarrow 0$ is again a renormalized solution.
\end{itemize}
\end{Lemma}
\begin{proof}
(I). From the equation \eqref{ro1}, one may take a regularizing sequence to obtain
\begin{equation}\label{ro2}
\partial_t S_m[\F]+\u\cdot\nabla S_m[\F]=S_m[\nabla\u\F]+r_m \quad\textrm{on}\quad \R^2\times (0,T),
\end{equation}
where $S_m$ are the standard smoothings with commutators
$$r_m=\u\cdot\nabla S_m[\F]-S_m[\u\cdot\nabla \F].$$
By \cite{DL}, one has $r_m\rightarrow 0$ in $L^1(\R^2\times(0,T))$ as $m\rightarrow 
\infty$. Following the strategy in \cite{DL} we can multiply the equation \eqref{ro2} by 
$b_{\F}(S_m[\F])$ and then pass to a limit with $m\rightarrow \infty$ to deduce the 
equation \eqref{l51} in the sense of distributions.

(II). Since $(\u^\dl,\F^\dl)$ are a sequence of renormalized solutions of \eqref{e1} in 
$\mathcal{D}'((0,T)\times\R^2)$, it follows that
\begin{equation*}
\begin{split}
 \partial_t T_k(\F^\dl_{ij})+\u^\dl\cdot\nabla T_k(\F^\dl_{ij})&=\Big(\nabla\u^\dl\F^\dl\Big)_{ij}T_k'(\F^\dl_{ij}).
\end{split}
\end{equation*}
Passing to the limit with $\dl\rightarrow 0$ in the above equation, we have
\begin{equation}\label{l52}
\partial_t \ov{T_k(\F_{ij})}+\u\cdot\nabla \ov{T_k(\F_{ij})}=\ov{\Big(\nabla\u\F\Big)_{ij}T_k'(\F_{ij})}\quad\textrm{in}\quad \mathcal{D}'(\R^2\times\R^+),
\end{equation}
where for any fixed $k>0$, one may observe following facts
$$\Big(\nabla\u^\dl\F^\dl\Big)_{ij}T_k'(\F^\dl_{ij})\rightarrow \ov{\Big(\nabla\u\F\Big)_{ij}T_k'(\F_{ij})}\quad \textrm{weakly in}\quad L^{2}(\O\times(0,T))$$
and
$$T_k(\F^\dl_{ij})\rightarrow \ov{T_k(\F_{ij})}\quad \textrm{weakly in}\quad C(0,T; L^p_{weak}(\O))\quad 1\le p<\infty$$ as $\dl\rightarrow\infty$.

Regularizing \eqref{l52} yields
\begin{equation}\label{l53}
\partial_t S_m\Big[\ov{T_k(\F_{ij})}\Big]+\u\cdot\nabla S_m\Big[\ov{T_k(\F_{ij})}\Big]=S_m\Big[\ov{\Big(\nabla\u\F\Big)_{ij}T_k'(\F_{ij})}\Big]+r_m,
\end{equation}
where $S_m$ are smoothing operators (mollifiers) in spatial variables with
$$r_m=\u\cdot\nabla S_m\Big[\ov{T_k(\F_{ij})}\Big]-S_m\Big[\u\cdot\nabla \ov{T_k(\F_{ij})}\Big].$$
It is easy to see that $r_m\rightarrow 0$ as $m\rightarrow \infty$ in 
$L^2(\R^2\times(0,T))$ for any fixed $k$, see for instance \cite{PL}.

Taking the dot products of both sides of the \eqref{l53} with 
$\nabla_{\F_{ij}}b\left(S_m\Big[\ov{T_k(\F)}\Big]\right)$, and then letting $m\rightarrow 
\infty$, one obtains
\begin{equation}\label{l54}
 \begin{split}
\partial_tb(\ov{T_k(\F)})+\u\cdot \nabla b(\ov{T_k(\F)})=\nabla_{\F_{ij}}b(\ov{T_k(\F)})\ov{\Big(\nabla\u\F\Big)_{ij}T_k'(\F_{ij})}
 \end{split}
\end{equation}
in $\mathcal{D}'(\O\times(0,T))$, where the notation $T_k(\F)$ stands for the matrix with 
entries given by $T_k(\F_{ij})$.

The next step is to take the limit of \eqref{l54} for $k\rightarrow\infty$. We start with 
the observation that
\begin{equation}\label{l57}
\ov{T_k(\F)}\rightarrow \F\quad \textrm{as}\quad k\rightarrow\infty \quad\textrm{in}\quad L^p(\O\times(0,T))\quad\textrm{for all}\quad 1\le p<3,
\end{equation}
as
\begin{equation*}
\begin{split}
\|\ov{T_k(\F)}-\F\|_{L^p(\O\times(0,T))}&\le \liminf_{\dl\rightarrow 0}\|T_k(\F^\dl)-\F^\dl\|_{L^p(\O\times(0,T))}\\
&\le 2\liminf_{\dl\rightarrow 0} k^{\f{p-3}{p}}\|\F^\dl\|_{L^3(\O\times(0,T))}^{3/p}\\
&\le 2ck^{\f{p-3}{p}}\rightarrow 0\quad \textrm{as}\quad k\rightarrow \infty.
\end{split}
\end{equation*}
Thus the desired identity \eqref{l51} will follow from \eqref{l54} by taking the limit for 
$k\rightarrow\infty$ provided one can verify
\begin{equation}\label{l55}
\nabla_{\F_{ij}}b(\ov{T_k(\F)})\ov{\Big(\nabla\u\F\Big)_{ij}T_k'(\F_{ij})}\rightarrow \nabla\u\F: \nabla_{\F}b(\F)\quad\textrm{in}\quad L^1(\O\times(0,T))
\end{equation}
when $k\rightarrow\infty$.

In order to verify \eqref{l55},
%we set $$D_{k,N}=\{(t,x)\in (0,T)\times\O: |\ov{T_k(\F)}|\le N\}.$$ Recall that $\nabla_{\F}b(\F)=0$ if $|\F|\ge N$.
we write
\begin{subequations}\label{l56}
 \begin{align}
&\nabla_{\F_{ij}}b(\ov{T_k(\F)})\ov{\Big(\nabla\u\F\Big)_{ij}T_k'(\F_{ij})}-\nabla\u\F: \nabla_{\F}b(\F)\nonumber\\
&\quad=\Big(\nabla_{\F_{ij}}b(\ov{T_k(\F)})-\nabla_{\F_{ij}}b(\F)\Big)\ov{\Big(\nabla\u\F\Big)_{ij}T_k'(\F_{ij})}\label{l56a}\\
&\qquad+\nabla_{\F_{ij}}b(\F)\left(\ov{\Big(\nabla\u\F\Big)_{ij}T_k'(\F_{ij})}-\Big(\nabla\u\F\Big)_{ij}\right).\label{l56b}
 \end{align}
\end{subequations}

We are now in position to verify weak convergences of each of these two terms above.
\texttt{Convergence of \eqref{l56a}.} First, the mean value theorem implies that
$$\nabla_{\F_{ij}}b(\ov{T_k(\F)})-\nabla_{\F_{ij}}b(\F)=\nabla_{\F}\nabla_{\F_{ij}}b(z):(\ov{T_k(\F)}-\F)$$ 
for some $z$ on the line segment connecting $\ov{T_k(\F)}$ and $\F$.
This identity along with \eqref{l57} imply 
$$\nabla_{\F_{ij}}b(\ov{T_k(\F)})\rightarrow \nabla_{\F_{ij}}b(\F) \quad \textrm{as}\quad k\rightarrow\infty \quad\textrm{in}\quad L^p(\O\times(0,T))\quad\textrm{for all}\quad 1\le p<3.$$
With the $L^\infty$ bound for $\nabla_\F b(\F)$, one concludes further that 
$$\nabla_{\F_{ij}}b(\ov{T_k(\F)})\rightarrow \nabla_{\F_{ij}}b(\F) \quad \textrm{as}\quad k\rightarrow\infty \quad\textrm{in}\quad L^p(\O\times(0,T))\quad\textrm{for all}\quad 1\le p<\infty.$$
Consequently, one has
\begin{equation*}
 \begin{split}
&\|\eqref{l56a}\|_{L^1(\O\times(0,T))}\\
&\quad\le \Big\|\nabla_{\F_{ij}}b(\ov{T_k(\F)})-\nabla_{\F_{ij}}b(\F)\Big\|_{L^6(\O\times(0,T))}\Big\|\ov{\Big(\nabla\u\F\Big)_{ij}T_k'(\F_{ij})}\Big\|_{L^{\f65}(\O\times(0,T))}\\
&\quad\le \Big\|\nabla_{\F_{ij}}b(\ov{T_k(\F)})-\nabla_{\F_{ij}}b(\F)\Big\|_{L^6(\O\times(0,T))}\limsup_{\dl\rightarrow 0}\Big\|\ov{\Big(\nabla\u^\dl\F^\dl\Big)_{ij}T_k'(\F^\dl_{ij})}\Big\|_{L^{\f65}(\O\times(0,T))}\\
&\quad\le \Big\|\nabla_{\F_{ij}}b(\ov{T_k(\F)})-\nabla_{\F_{ij}}b(\F)\Big\|_{L^6(\O\times(0,T))}\limsup_{\dl\rightarrow 0}\Big\|\nabla\u^\dl\Big\|_{L^2(\O\times(0,T))}\Big\|\F^\dl\Big\|_{L^{3}(\O\times(0,T))}\\
&\quad\le c\Big\|\nabla_{\F_{ij}}b(\ov{T_k(\F)})-\nabla_{\F_{ij}}b(\F)\Big\|_{L^6(\O\times(0,T))}\rightarrow 0\quad \textrm{as}\quad k\rightarrow \infty.
 \end{split}
\end{equation*}

\texttt{Convergence of \eqref{l56b}.} By the constraint $\Dv(\F^\dl)^\top=0$, one 
checks that the quadratic 
term $\nabla\u^\dl\F^\dl\rightarrow \nabla\u\F$ converges weakly in $L^{\f65}(\O\times 
(0,T))$ as $\dl\rightarrow 0$. This in turn leads to
\begin{equation*}
 \begin{split}
\|\eqref{l56b}\|_{L^1(\O\times(0,T))}
&\le \sup_{0\le |\F|\le N}|\nabla_{\F}b(\F)|\liminf_{\dl\rightarrow 0}\left\|\Big(\nabla\u^\dl\F^\dl\Big)_{ij}T_k'(\F^\dl_{ij})-\Big(\nabla\u^\dl\F^\dl\Big)_{ij}\right\|_{L^1(\O\times(0,T))}\\
&\le c\limsup_{\dl\rightarrow 0}\|\nabla\u^\dl\F^\dl\|_{L^{\f65}(\O\times(0,T))}\liminf_{\dl\rightarrow 0}\left\|T_k'(\F^\dl_{ij})-1\right\|_{L^6(\O\times(0,T))}\\
&\le c\left(\liminf_{\dl\rightarrow 0} \left|\Big\{(x,t)\in \O\times(0,T): |\F^\dl|\ge k\Big\}\right|\right)^{\f16}\\
&\le ck^{-\f12}\limsup_{\dl\rightarrow 0}\|\F^\dl\|^{\f12}_{L^3(\O\times(0,T))}\rightarrow 0\quad\textrm{as}\quad k\rightarrow \infty,
 \end{split}
\end{equation*}
where $|\cdot|$ denotes the Lebesgue measure.
\end{proof}

As a consequence of Lemma \ref{l5}, by choosing $b(\F)=T_k(|\F|)$ in \eqref{l51}, we have
\begin{equation}\label{l51a}
\partial_t T_k(|\F|)+\u\cdot\nabla T_k(|\F|)=\nabla\u\F: \F T_k'(|\F|)|\F|^{-1}=\nabla\u: \F\F^\top T_k'(|\F|)|\F|^{-1}.
\end{equation}

\subsection{Strong convergence of $\F$}

The other key ingredient in the proof of Theorem \ref{mt} is the strong convergence of 
$\F$ which we shall discuss now.
\begin{Proposition}\label{l7}
With the same notations as above, we have 
$$\F^\i\rightarrow \F\quad\textrm{in}\quad L^2_{loc}(\R^2\times(0,T))$$
as $\i\rightarrow 0$.
\end{Proposition}
\begin{proof}
Since $(\u^\i,\F^\i)$ are solutions of \eqref{e21}, we multiply the second equation of 
\eqref{e21} by $\F^\i T'_k(|\F^\i|)|\F^\i|^{-1}$ to obtain
\begin{equation}\label{l7a}
\partial_tT_k(|\F^\i|)+\u\cdot\nabla T_k(|\F^\i|)=\nabla\u^\i:\F^\i(\F^\i)^\top T_k'(|\F^\i|)|\F^\i|^{-1}.
\end{equation}

Passing to the limit in \eqref{l7a} with $\i\rightarrow 0$ yields
\begin{equation*}
\partial_t\ov{T_k(|\F|)}+\u\cdot\nabla \ov{T_k(|\F|)}=\ov{\nabla\u:\F\F^\top T_k'(|\F|)|\F|^{-1}}
\end{equation*}
in the sense of distributions.
This combines with \eqref{l51a} imply that
\begin{equation*}
\begin{split}
&\partial_t\Big(\ov{T_k(|\F|)}-T_k(|\F|)\Big)+\u\cdot\nabla \Big(\ov{T_k(|\F|)}-T_k(|\F|)\Big)\\
&\quad=\Big(\ov{\nabla\u:\F\F^\top T_k'(|\F|)|\F|^{-1}}-\nabla\u: \F\F^\top T_k'(|\F|)|\F|^{-1}\Big).
\end{split}
\end{equation*}
Integrating over $\R^2\times(0,t)$ and using Lemma \ref{l31}, one obtains
\begin{equation}\label{16}
\begin{split}
&\int_{\R^2}\Big(\ov{T_k(|\F|)}-T_k(|\F|)\Big)(t)dx\\
&\quad=\int_0^t\int_{\R^2}\Big(\ov{\nabla\u:\F\F^\top T_k'(|\F|)|\F|^{-1}}-\nabla\u: \F\F^\top T_k'(|\F|)|\F|^{-1}\Big)dxds\\
&\quad=\lim_{\i\rightarrow 0}\int_0^t\int_{\R^2}\nabla\u^\i:\Big(\F^\i(\F^\i)^\top T'_k(|\F^\i|)|\F^\i|^{-1}-\F\F^\top T'_k(|\F|)|\F|^{-1}\Big)dxds.
\end{split}
\end{equation}

Let $h(\F)=\F\F^\top T_k'(|\F|)|\F|^{-1}$ for all nonzero matrix $\F$, it 
follows easily that
$$|\nabla_{\F}h(\F)|\le c$$
for some constant $c>0$ which is independent of $k$. Next, the mean-value theorem implies, for  
$1\le i,j\le 2$, that
\begin{equation*}
\begin{split}
&(\F^\i(\F^\i)^\top)_{ij} T'_k(|\F^\i|)|\F^\i|^{-1}-(\F\F^\top)_{ij} T'_k(|\F|)|\F|^{-1}\\
&\quad=h_{ij}(T_k(\F^\i))-h_{ij}(T_k(\F))\\
&\quad=\nabla_{\F} h_{ij}(G): (T_k(\F^\i)-T_k(\F))
\end{split}
\end{equation*}
for some matrix $G$ between $T_k(\F^\i)$ and $T_k(\F)$. Hence the right hand side of \eqref{16} can be estimated as
\begin{equation}\label{161}
\begin{split}
&\left|\lim_{\i\rightarrow 0}\int_0^t\int_{\R^2}\nabla\u^\i:\Big(\F^\i(\F^\i)^\top T'_k(|\F^\i|)|\F^\i|^{-1}-\F\F^\top T'_k(|\F|)|\F|^{-1}\Big)dxds\right|\\
&\quad= c\limsup_{\i\rightarrow 0}\Big(\|\nabla\u^\i\|_{L^2(\R^2\times\R^+)}\|T_k(\F^\i)-T_k(\F)\|_{L_{loc}^2(\R^2\times \R^+)}\Big)\\
&\quad= c\limsup_{\i\rightarrow 0}\|T_k(\F^\i)-T_k(\F)\|_{L_{loc}^2(\R^2\times \R^+)}\\
&\quad= c\limsup_{\i\rightarrow 0}\|T_k(\F^\i)-\F\|_{L_{loc}^2(\R^2\times \R^+)}+\limsup_{\i\rightarrow 0}\|T_k(\F)-\F\|_{L_{loc}^2(\R^2\times \R^+)}.
\end{split}
\end{equation}
In view of \eqref{l57}, terms on right hand side of \eqref{161} tend 
to zero as $k\rightarrow \infty$, and therefore 
\begin{equation}\label{161a}
\lim_{k\rightarrow\infty}\int_{\R^2}\Big(\ov{T_k(|\F|)}-T_k(|\F|)\Big)(t)dx= 0.
\end{equation}

On the other hand, we have
\begin{equation*}
\int_{\R^2}\Big(\ov{T_k(|\F|)}-T_k(|\F|)\Big)(t)dx\rightarrow \int_{\R^2}\Big(\ov{|\F|}-|\F|\Big)(t)dx
\end{equation*}
as $k\rightarrow \infty$. Thus there follows from \eqref{161a} that
\begin{equation}\label{161b}
\begin{split}
\int_{\R^2}\Big(\ov{|\F|}-|\F|\Big)(t)dx= 0.
\end{split}
\end{equation}
The convexity of $|\F|$ implies 
$$\ov{|\F|}\ge |\F|\quad \textrm{almost everywhere}.$$ 
This combines with \eqref{161b} lead to
$$\ov{|\F|}=|\F|\quad \textrm{almost everywhere}.$$
Hence $\F^\i$ converges to $\F$ almost everywhere as $\i\rightarrow 0$. 
This pointwise convergence and the higher integrability estimate for $\F$ 
in $L^3$ yield
\begin{equation*}
\F^\i\rightarrow \F\quad\textrm{in}\quad L^2_{loc}(\R^2\times(0,T))
\end{equation*}
as $\i\rightarrow 0$.
\end{proof}

%With the strong convergence of $\F^\i$ in $L^2$, constraints \eqref{c} and \eqref{c1} are %preserved for the limit function $\F$ in the sense of distributions.

\bigskip\bigskip

%%%%%%%%%%%%%%%%%%%%%%%%%%%%%%%%%%%%%%%%%%%%%%%%%%%%%%%%%%%%%%%%%%%%%%%%%%%%%%%%%%%%%%%%%

\section{Approximations and Proof of Theorem \ref{mt}}

In this section we shall present a construction of approximate solutions of 
\eqref{e1}, and the a proof of Theorem \ref{mt}. Let us introduce the 
following approximate system: 
\begin{equation}\label{aae1}
\begin{cases}
\partial_t\u+\u\cdot\nabla\u-\D\u+\nabla P=\Dv\Big(\F\F^\top+\dl|\F-I|^2\Big[(\F-I)\F^\top+\F(\F-I)^\top\Big]\Big)\\
\partial_t\F+\u\cdot\nabla\F=\nabla\u\F+\eta\D\F\\
(\u,\F)|_{t=0}=(\u_0,\F_0)\quad\textrm{and}\quad \Dv\u=0
\end{cases}
\end{equation}
with two independent positive parameters $\eta>0$ and $\dl>0$. Here the parameter 
$\eta$ in the second equation of \eqref{aae1} is needed when we apply the 
classical Galerkin method. However, the presence of $\eta\D\F$ in the second 
equation of \eqref{aae1} eliminates the conserved quantity $\det\F=1$, and hence 
one needs a necessary modification of the estimates in Section 2. The parameter 
$\dl$ in the first equation of \eqref{aae1} is used to control the estimate of 
$\det\F$. This approximation of the stored elastic energy (hence also the 
original Oldroyd model) does not change the underlying physics and the structure 
of the equations, and it turns out to be rather useful as otherwise the model 
is simply too rigid for analysis.

Denote the solutions of \eqref{aae1} by $(\u_{\dl, \eta},\F_{\dl, \eta}, P_{\dl, \eta})$ which are constructed via the classical Galerkin method. 
The proof of Theorem \ref{mt} can be divided into two steps: taking the 
vanishing viscosity limit $\eta\rightarrow 0$ and, taking the vanishing of 
(artificially) modified elastic energy limit $\dl\rightarrow 0$.  
In both steps, the arguments in previous sections may be applied with 
some necessary changes.  The most noticeable change is in the first step 
$\eta\rightarrow 0$. Let us describe the proofs below.

Applying the divergence operator to the second equation in \eqref{aae1}, one obtains
$$\partial_t\Dv\F^\top+\u\cdot\nabla\Dv\F^\top=\eta\D\Dv\F^\top.$$
Combining with $\Dv(\F_0^\top)=0$, it implies that the identity \eqref{c} for 
all positive $t>0$, that is, $$\Dv(\F_{\dl,\eta})^\top=0\quad\textrm{for 
all}\quad t\ge 0.$$

Though we no longer have $\det\F_{\dl,\eta}=1$, the arguments in Section 2 
for the higher integrability can be easily modified. Indeed, the energy 
law yields a uniform bound (independent of small parameter $\eta$) 
on $\|\det\F_{\dl,\eta}\|_{L^\infty(0,T; L^{2}(B_a(0)))}$. 
To be more precise, we proceed our 
proofs as follows. 

We decompose the symmetric matrix as in \eqref{200} so that the stress tensor
\begin{equation*}
\begin{split} 
\tau_{\dl,\eta}&=\F_{\dl,\eta}\F_{\dl,\eta}^\top+\tau^1_{\dl,\eta}\\
&=\F_{\dl,\eta}\F_{\dl,\eta}^\top + \dl |\F_{\dl,\eta}-I|^{2}\Big[(\F_{\dl,\eta}-I)\F_{\dl,\eta}^\top+\F_{\dl,\eta}(\F_{\dl,\eta}-I)^\top\Big]
\end{split}
\end{equation*} 
can be expressed in terms of $\Pi_i^*$ $(i=1,2,3).$ Here, $\Pi_i^* =
(1 + \delta|\F_{\dl,\eta} - I|^2)\Pi_i + S_i$, and where
$-\delta|\F_{\dl,\eta} - I|^2(\F_{\dl,\eta} + \F_{\dl,\eta}^\top)$ is
decomposed in the same way in terms of $S_i$ $(i=1,2,3).$ We observe
the energy law of \eqref{aae1}:
\begin{equation*}
\begin{split}
&\f12\f{d}{dt}\int_{\R^2}\Big(|\u|^2+|\F-I|^2+\dl|\F-I|^4\Big)dx\\
&\quad+\int_{\R^2}\Big(|\nabla\u|^2+\eta|\nabla\F|^2+2\dl\eta||\F-I|\nabla\F|^2+2\dl\eta||\F-I|\nabla|\F-I||^2\Big)dx\le 0.
\end{split}
\end{equation*}

It implies that $S_i$'s are in $L^{4/3}_{loc}$ with norms of order
$O(\delta^{1/4})$, for $i = 1,2,3$. For all practical purposes and for
our proofs, we can simply ignore these terms $S_i$ $(i=1,2,3)$. We
note that $\Pi_i$ $(i=1,2,3)$ corresponds to decomposition of
$\F_{\dl,\eta}\F_{\dl,\eta}^\top$. Applying the equation satisfied by
$\Pi_i (i=2,3)$ as in Section 2 with $F = \F_{\dl,\eta}$ and using the
same representations for $\Pi_i^*$'s in terms of the left-hand side of
\eqref{aae1}, we lead to one of the key estimates
\begin{equation}\label{ae40}
(1 + \delta |\F_{\dl,\eta} - I|^2)(|\Pi_2| + |\Pi_3|)^{2} \in
L_{loc}^1 (\mathbb{R}^2 \times \mathbb{R}_+)
\end{equation}

From the energy law, one has a uniform bound depending only on
the parameter $\dl$ that
\[
|\F_{\dl,\eta}| \in L^4_{loc} (\mathbb{R}^2 \times \mathbb{R}_+).
\]
We \underline{Claim} that the arguments in Sections 3 and 4 can be applied 
so that
$\F_{\dl,\eta} \to \F_{\dl}$ as $\eta \to 0$ strong in $L^2_{loc}$, see 
below for some necessary modifications. 

One observes further from \eqref{ae40}
   \iffalse
\[
(1 + \delta |\F_{\dl,\eta} - I|^2)(|\Pi_2| + |\Pi_3|)^{2} \in 
L_{loc}^1 (\mathbb{R}^2 \times \mathbb{R}_+),
\]\fi
and from the Holder's inequality as well as the energy law for (5.1) that, 
the traceless part of $\tau^1_{\dl,\eta}$ in $L^p_{loc}(\R^2\times\R^+)$ 
for some $1<p\le 4/3$ (any $1<p\le 4/3$ will work). 
This bound implies also that the total pressure, up to a constant, 
of the system (5.1) is also uniformly bounded in $\dl, \eta$ in the same spaces 
$L^p_{loc}(\R^2\times\R^+)$ with $1<p\le\f43$. One thus can show the 
convergence of the traceless part of $\tau_{\dl,\eta}$ as
$\eta$ goes to zero in these $L^p_{loc}$ spaces.

Let us derive next an equation satisfied by $\det\F_{\dl,\eta}$.
Since $\f{d\det\F}{d\F}=\det\F\F^{-\top}$ (see \cite{LEILIUZHOU}) , one deduces from the second equation of \eqref{e21} that
\begin{equation}\label{ae15}
\partial_t\det\F+\u\cdot\nabla\det\F=\eta\det\F\F^{-\top}:\D\F.
\end{equation}
Here we used $$\det\F\F^{-\top}:\nabla\u\F=\sum_{i,j,k=1}^2\det\F\F^{-1}_{ji}\partial_{k}\u_i\F_{kj}=\det\F\sum_{k=1}^2\partial_{x_k}\u_k=0.$$
To see the next set of computations more clearly, we start with the fact
$\Dv\F^\top=0$. Hence there exists a vector-valued function $\phi=(\phi_1,\phi_2)$ such that
$$\F=\Big(
\begin{array}{ccc}
-\partial_{x_2}\phi_1\quad -\partial_{x_2}\phi_2\\
\partial_{x_1}\phi_1\quad \partial_{x_1}\phi_2
\end{array}\Big).$$
A direct computation shows 
$\det\F\F^{-\top}:\D\F=\D\det\F-2\sum_{i=1}^2\det\partial_{x_k}\F,$ and
\begin{equation*}
\begin{split}
\det\partial_{x_k}\F&=\f12\partial_{x_2}\Big[-\partial_{x_k}\phi_1\partial_{x_1}\partial_{x_k}\phi_2+\partial_{x_k}\partial_{x_1}\phi_1\partial_{x_k}\phi_2\Big]\\
&\quad+\f12\partial_{x_1}\Big[-\partial_{x_k}\partial_{x_2}\phi_1\partial_{x_k}\phi_2+\partial_{x_k}\partial_{x_1}\phi_2\partial_{x_k}\phi_2\Big].
\end{split}
\end{equation*}
Therefore the equation \eqref{ae15} yields
\begin{equation}\label{ae16}
\begin{split}
\partial_t\det\F+\u\cdot\nabla\det\F-\eta\D\det\F&=-\eta\sum_{k=1}^2\Big\{\partial_{x_2}\Big[-\partial_{x_k}\phi_1\partial_{x_1}\partial_{x_k}\phi_2+\partial_{x_k}\partial_{x_1}\phi_1\partial_{x_k}\phi_2\Big]\\
&\quad+\partial_{x_1}\Big[-\partial_{x_k}\partial_{x_2}\phi_1\partial_{x_k}\phi_2+\partial_{x_k}\partial_{x_1}\phi_2\partial_{x_k}\phi_2\Big]\Big\}.
%&\overset{def}=\eta\Dv Q(\F,\nabla\F).
\end{split}
\end{equation}

Notice that the right hand side of the above equation is a divergence of a bilinear quantity involving $\F$ and $\nabla\F$, due to the structure of determinants.
Applying a standard argument for parabolic equations, one can obtain 
the following uniform estimates in $\eta$ for solutions of \eqref{ae16}: 
\begin{equation}\label{aae3}
\begin{split}
&\|\det\F-1\|_{L^\infty(0,T;L^{2}(\R^2))}+
\sqrt{\eta}\|\nabla(\det\F-1)\|_{L^{2}((0,T)\times \R^2)}\\
&\le C\sqrt{\eta}\|\F\nabla\F\|_{L^{2}((0,T)\times 
\R^2)} \le C(\dl).
\end{split}
\end{equation}
It implies also
$$\|\det\F\|_{L^\infty(0,T; L^{2}(B_a(0)))}\le C(a,\dl)\Big(\|\det\F-1\|_{L^\infty(0,T; L^2(\R^2))}+1\Big).$$

We explain now how to verify our \underline{Claim} by modifying the 
earlier arguments, that is the 
strong convergence  of $\F_{\dl,\eta}$ to $\F_{\dl}$ as $\eta\rightarrow 0$ in 
$L^2_{loc}$.  We follow the proof of Proposition \ref{l5}, and some 
modifocations for \eqref{l7a} are needed.  The equation \eqref{l7a} is now 
replaced by
\begin{equation*}
\begin{split}
\partial_tT_k(|\F_{\dl,\eta}|)+\u\cdot\nabla T_k(|\F_{\dl,\eta}|)&=\nabla\u^\i:\F_{\dl,\eta}\F_{\dl,\eta}^\top T_k'(|\F_{\dl,\eta}|)|\F_{\dl,\eta}|^{-1}\\
&\quad+\i\D\F_{\dl,\eta}:\F_{\dl,\eta}\F_{\dl,\eta}^\top T_k'(|\F_{\dl,\eta}|)|\F_{\dl,\eta}|^{-1}\\
&\le \nabla\u^\i:\F_{\dl,\eta}\F_{\dl,\eta}^\top T_k'(|\F_{\dl,\eta}|)|\F_{\dl,\eta}|^{-1}\\
&\quad+\eta\Dv\Big[\nabla\F_{\dl,\eta}:\F_{\dl,\eta} T'_k(|\F_{\dl,\eta}|)|\F_{\dl,\eta}|^{-1}\Big].
\end{split}
\end{equation*}
Here in the last inequality one has applied the following:
\begin{equation*}
\begin{split}
|\nabla|\F||^2&=\sum_{k=1}^2\Big(\sum_{i,j=1}^2 |\F|^{-1}\F_{ij}\partial_k\F_{ij}\Big)^2=\sum_{k=1}^2|\F|^{-2}\Big(\sum_{i,j=1}^2\F_{ij}\partial_k\F_{ij}\Big)^2\\
&\le \sum_{k=1}^2|\F|^{-2}\Big(\sum_{i,j=1}^2\F_{ij}^2\Big)\Big(\sum_{i,j=1}^2(\partial_k\F_{ij})^2\Big)=\sum_{i,j,k=1}^2(\partial_k\F_{ij})^2=|\nabla\F|^2,
\end{split}
\end{equation*}
and consequently
\begin{equation*}
\begin{split}
\D\F:\F T'_k(|\F|)|\F|^{-1}&=\Dv\Big[\nabla\F:\F T'_k(|\F|)|\F|^{-1}\Big]-\nabla\F:\nabla\F T_k'(|\F|)|\F|^{-1}\\
&\quad-\nabla_l\F:\F T_k''(|\F|)\nabla_l(|\F|)|\F|^{-1}+\nabla_l\F:\F T_k'(|\F|)\nabla_l(|\F|)|\F|^{-2}\\
&=\Dv\Big[\nabla\F:\F T'_k(|\F|)|\F|^{-1}\Big]-|\nabla\F|^2 T_k'(|\F|)|\F|^{-1}\\
&\quad- T_k''(|\F|)|\nabla |\F||^2+\eta T_k'(|\F|)|\nabla |\F||^2|\F|^{-1}\\
&\le \Dv\Big[\nabla\F:\F T'_k(|\F|)|\F|^{-1}\Big]\quad\textrm{in}\quad \mathcal{D}'(\R^2\times(0,T))
\end{split}
\end{equation*}
as $T_k'\ge 0$ and $T''_k=0$ almost everywhere. 
Moreover one observes that the term 
$\eta\Dv\Big[\nabla\F_{\dl,\eta}:\F_{\dl,\eta} 
T'_k(|\F_{\dl,\eta}|)|\F_{\dl,\eta}|^{-1}\Big]$ converges to $0$ in the sense of 
distributions as $\eta\rightarrow 0$. This completes the proof of our claim.

As a consequence of the strong convergence of $\F_{\dl,\eta}$ to $\F_{\dl}$ and the 
improved integrability for the traceless part of $\tau^1_{\dl,\tau}$, one 
deduces the traceless part of $\tau_{\dl,\eta}$ converges to $\tau_\dl$ in 
$L^1_{loc}$ as illustrated above, where $\tau_\dl$ is 
given by
$$\tau_{\dl}=\F_{\dl}\F_{\dl}^\top + \dl |\F_{\dl}-I|^{2}\Big[(\F_{\dl}-I)\F_{\dl}^\top+\F_{\dl}(\F_{\dl}-I)^\top\Big].$$
As we shown earlier that the strong convergence of $\F_{\dl,\eta}$ implies that the limiting $\F_\dl$ 
satisfies the constraints \eqref{c} and \eqref{c1}. Furthermore, the identity 
$\det\F_{\dl}=1$ follows from the equation
$$\partial_t\det\F_{\dl}+\u_{\dl}\cdot\nabla\det\F_{\dl}=0$$
in the sense of distributions.

To summarize, after passing to the limit $\eta\rightarrow 0$, we arrive at the following 
system
\begin{equation}\label{aae2}
\begin{cases}
\partial_t\u+\u\cdot\nabla\u-\D\u+\nabla P=\Dv\Big(\F\F^\top+\dl|\F-I|^2\Big[(\F-I)\F^\top+\F(\F-I)^\top\Big]\Big)\\
\partial_t\F+\u\cdot\nabla\F=\nabla\u\F\\
(\u,\F)|_{t=0}=(\u_0,\F_0)\quad\textrm{and}\quad \Dv\u=0.
\end{cases}
\end{equation}
Moreover the solution of \eqref{aae2} satisfies the Piola's identity \eqref{c1}, $\det\F=1$, and $\Dv\F^\top=0.$

Finally with the identity $\det\F_{\dl}=1$, one can then take the limit with 
$\dl\rightarrow 0$. Following the exactly same arguemtns as that in 
Sections 2, 3 and 4, one obtains a limit $(\u,\F)$ which is a global weak 
solution of \eqref{e1} with finite energy, and hence the proof of Theorem 
\ref{mt} is completed.
\bigskip\bigskip

%%%%%%%%%%%%%%%%%%%%%%%%%%%%%%%%%%%%%%%%%%%%%%%%%%%%%%%%%%%%%%%%%%%%%%%%%%%%%%%%%%%%%%%%%

\section{Application of Decomposition \eqref{200}}

The decomposition \eqref{200} that uses the symmetry of $\tau$ 
is also related in dimensions two to the so-called Hopf differential 
for the flow map. The latter has being extremely useful to study mappings in 
differential geometry.  Here we want to show that it could be used  
to derive one of the key estimates in our previous work \cite{HL} 
for the construction of global weak solutions near the equilibrium.

In view of these five observations in Section 2.1, we introduce five variants of the \textit{effective viscous flux} as
\begin{itemize}
\item $\mathcal{G}_1=\Cu\u-(-\D)^{-1}\Big[2\partial_1\partial_2\Pi_2+(\partial_2^2-\partial_1^2)\Pi_3\Big],$
\item $\tilde{\mathcal{G}_1}=\sqrt{2}\Cu\u+(-\D)^{-1}\Big[(\partial_1^2-\partial_2^2)\Pi_2+2\partial_1\partial_2\Pi_3\Big],$
\item $\hat{\mathcal{G}_1}=\sqrt{2}\Cu\u-\hat{P},$
\item $\mathcal{G}_2=\sqrt{2}\Big(\partial_1\u_2+\partial_2\u_1\Big)+\Big(\partial_1\u_1-\partial_2\u_2\Big)+\Pi_2,$
\item $\mathcal{G}_3=\Big(\partial_1\u_2+\partial_2\u_1\Big)-\sqrt{2}\Big(\partial_1\u_1-\partial_2\u_2\Big)+\Pi_3.$
\end{itemize}
Moreover, according to their definitions, it is easy to observe that
$$\|\D\mathcal{G}_1\|_{L^2}+\|\D\mathcal{G}_2\|_{L^2}+\|\D\mathcal{G}_3\|_{L^2}\le C\|\nabla \dot{\u}\|_{L^2}.$$
One of advantages of using $\mathcal{G}_2$ and $\mathcal{G}_3$ is due to the fact that 
these quantities do not involve singular operators and, hence they are suitable for 
$L^\infty$ estimates.

Using the same notation as in \cite{HL}, we denote by $E$ a small perturbation of $F$ 
around the identity matrix; that is, $E=\F-I$. We shall explain how
$\mathcal{G}_2$ and $\mathcal{G}_3$ can be used to give  $L^\infty$ estimates of 
$E_{11}-E_{22}$ and 
$E_{12}+E_{21}$ which are linear parts of $\Pi_2$ and $\Pi_3$ respectively. Indeed, 
form the equation of $\F$, some careful computations yield
\begin{equation}\label{ap1}
\begin{split}
\f{d}{dt}\left(E_{11}-E_{22}+\sqrt{2}(E_{12}+E_{21})\right)&=\partial_1\u_1-\partial_2\u_2+\sqrt{2}(\partial_1\u_2+\partial_2\u_1)\\
&\quad+\f12\Big[\partial_1\u_1-\partial_2\u_2+\sqrt{2}(\partial_2\u_1+\partial_1\u_2)\Big](E_{11}+E_{22})\\
&\quad+\f12\Big[\partial_1\u_2+\partial_2\u_1-\sqrt{2}(\partial_1\u_1-\partial_2\u_2)\Big](E_{21}-E_{12})\\
&\quad+\f12(\partial_2\u_1-\partial_1\u_2)\Big[E_{21}+E_{12}+\sqrt{2}(E_{22}-E_{11})\Big],
\end{split}
\end{equation}
and
\begin{equation}\label{ap2}
\begin{split}
\f{d}{dt}\left(E_{12}+E_{21}-\sqrt{2}(E_{11}-E_{22})\right)&=\partial_1\u_2+\partial_2\u_1-\sqrt{2}(\partial_1\u_1-\partial_2\u_2)\\
&\quad+\f12\Big[\partial_2\u_1+\partial_1\u_2-\sqrt{2}(\partial_1\u_1-\partial_2\u_2)\Big](E_{11}+E_{22})\\
&\quad-\f12\Big[\partial_1\u_1-\partial_2\u_2+\sqrt{2}(\partial_2\u_1+\partial_1\u_2)\Big](E_{21}-E_{12})\\
&\quad-\f12(\partial_2\u_1-\partial_1\u_2)\Big[E_{11}-E_{22}+\sqrt{2}(E_{21}+E_{12})\Big],
\end{split}
\end{equation}
where $\f{d}{dt}$ stands for the material derivative $\partial_t+\u\cdot\nabla$. 
One notices that
\begin{equation*}
\begin{split}
&\left(E_{12}+E_{21}-\sqrt{2}(E_{11}-E_{22})\right)(E_{12}+E_{21})+\left(E_{11}-E_{22}+\sqrt{2}(E_{12}+E_{21})\right)(E_{11}-E_{22})\\
&\quad=|E_{12}+E_{21}|^2+|E_{11}-E_{22}|^2.
\end{split}
\end{equation*}
Consider the following quantity 
$\eqref{ap1}\times\Big[E_{11}-E_{12}+\sqrt{2}(E_{12}+E_{21})\Big]+\eqref{ap2}\times\Big[E_{12}+E_{21}-\sqrt{2}(E_{11}-E_{22})\Big]$, 
one calculates further that
\begin{equation}\label{ap3}
\begin{split}
&3\f{d}{dt}\left(|E_{12}+E_{21}|^2+|E_{11}-E_{22}|^2\right)+|E_{12}+E_{21}|^2+|E_{11}-E_{22}|^2\\
&\quad=\Big[\mathcal{G}_2-\Pi_2+(E_{11}-E_{22}\Big]\Big[E_{11}-E_{22}+\sqrt{2}(E_{12}+E_{21})\Big]\\
&+\f12(\mathcal{G}_2-\Pi_2)(E_{11}+E_{22})\Big[E_{11}-E_{22}+\sqrt{2}(E_{12}+E_{21})\Big]\\
&+\f12(\mathcal{G}_3-\Pi_3)(E_{21}-E_{12})\Big[E_{11}-E_{22}+\sqrt{2}(E_{12}+E_{21})\Big]\\
&+\Big[\mathcal{G}_3-\Pi_3+(E_{12}+E_{21}\Big]\Big[E_{12}+E_{21}-\sqrt{2}(E_{11}-E_{22})\Big]\\
&+\f12(\mathcal{G}_3-\Pi_3)(E_{11}+E_{22})\Big[E_{12}+E_{21}-\sqrt{2}(E_{11}-E_{22})\Big]\\
&-\f12(\mathcal{G}_2-\Pi_2)(E_{21}-E_{12})\Big[E_{12}+E_{21}-\sqrt{2}(E_{11}-E_{22})\Big].
\end{split}
\end{equation}

We also observe that
$$\Pi_2-(E_{11}+E_{22})=\f12(E_{11}+E_{22})(E_{11}-E_{22})+\f12(E_{12}+E_{21})(E_{12}-E_{21}),$$
and that 
$$\Pi_3-(E_{12}+E_{21})=E_{21}(E_{11}-E_{22})+E_{22}(E_{21}+E_{12}).$$
Hence the right hand side of \eqref{ap3} can be controlled by
$$\Big[|\mathcal{G}_2|+|\mathcal{G}_3|\Big]\Big(|E_{12}+E_{21}|+|E_{11}-E_{22}|\Big)+\eta \Big(|E_{12}+E_{21}|^2+|E_{11}-E_{22}|^2\Big)$$
whenever $\|E\|_{L^\infty}\le \eta$. Substituting this back to \eqref{ap3} yields
\begin{equation*}
\begin{split}
&3\f{d}{dt}\left(|E_{12}+E_{21}|^2+|E_{11}-E_{22}|^2\right)+\f12\Big(|E_{12}+E_{21}|^2+|E_{11}-E_{22}|^2\Big)\\
&\le C\left(|\mathcal{G}_2|^2+|\mathcal{G}_3|^2\right).
\end{split}
\end{equation*}
The latter, by an integration along the trajectory (cf. Lemma 4.1 in \cite{HL} with 
$\mathfrak{G}$ replaced by $\mathcal{G}_2$ and $\mathcal{G}_3$ as above), implies the 
$L^\infty$ bounds of $E_{11}-E_{22}$ and $E_{12}+E_{21}$ for all $t$. In addition, 
applying the same arguments as in the proof of the Lemma \ref{l4a}, one obtains
$$|E|^2\le C\left(\left||E_1|^2-|E_2|^2\right|+|E_1\cdot E_2|+|\det E|\right),$$
and consequently 
$$|E|^2\le C\left(|E_{11}-E_{22}|^2+|E_{12}+E_{21}|^2+|\det E|\right).$$
The final estimate along with the quadratic property of $\det E$, lead to the 
desired estimate for $\|E(t)\|_{L^\infty}$ in terms of $\mathcal{G}_j's$. 
The $\|E(t)\|_{L^\infty}$ bound is a crucial estimate in the section 4 of \cite{HL}. 
Moreover, the decomposition \eqref{200} has also an advantage that it 
provides pointwise estimates in physical variables by getting rid of 
zero-th order singular operators in $\mathcal{G}_1$.

\bigskip\bigskip

%%%%%%%%%%%%%%%%%%%%%%%%%%%%%%%%%%%%%%%%%%%%%%%%%%%%%%%%%%%%%%%%%%%%%%%%%%%%%%%%%%%%%%%%%%

\section*{Acknowledgement}
The work of Xianpeng Hu is partially supported by the start-up grant from City University of Hong Kong and the ECS grant 9048035. The work of Fanghua Lin is partially supported by the NSF grant DMS-1501000.

\bigskip\bigskip

%%%%%%%%%%%%%%%%%%%%%%%%%%%%%%%%%%%%%%%%%%%%%%%%%%%%%%%%%%%%%%%%%%%%%%%%%%%%%%%%%%%%%%%%%

\end{document}